\documentclass[11pt]{article}

\usepackage[utf8]{inputenc}
\usepackage[T1]{fontenc}
\usepackage{microtype}

\usepackage{geometry}
\usepackage{amsmath, amssymb, array, bm, mathrsfs, bbm, cancel, mhchem}

\usepackage{graphicx, float}
\usepackage[table]{xcolor}
\usepackage{colortbl, booktabs, multirow, tcolorbox, etoolbox}

\usepackage{import, enumitem, lineno, soul, todonotes}
\usepackage[normalem]{ulem}

\usepackage[acronym, toc, sort=use, nonumberlist]{glossaries}
\preto\printglossary{\glsaddall}
\makeglossaries
\usepackage{nomencl}
\makenomenclature

\usepackage{listings}
\lstdefinestyle{mathematica}{
    language=Mathematica,
    basicstyle=\small\ttfamily,
    backgroundcolor=\color{gray!10},
    frame=single,
    breaklines=true,
    commentstyle=\color{green!60!black},
    keywordstyle=\color{blue},
    stringstyle=\color{red},
    showstringspaces=false,
    tabsize=2
}

\usepackage{tikz, tikz-cd}
\usepackage{pgfplots}
\pgfplotsset{compat=1.15}
\usetikzlibrary{
    shapes, shapes.geometric, shapes.symbols, shapes.callouts,
    arrows, arrows.meta,
    positioning, calc,
    decorations.pathmorphing, decorations.markings,
    backgrounds, fit,
    patterns, graphs, automata
}
\tikzset{
    startstop/.style={rectangle, rounded corners, minimum width=3cm,
        minimum height=1cm, text centered, draw=black, fill=blue!30},
    io/.style={trapezium, trapezium left angle=70, trapezium right angle=110,
        minimum width=3cm, minimum height=1cm, text centered,
        draw=black, fill=blue!30},
    process/.style={rectangle, minimum width=3cm, minimum height=1cm,
        text centered, draw=black, fill=green!30},
    arrow/.style={thick, ->, >=stealth},
    cloud/.style={draw, ellipse, fill=red!20, node distance=0.87cm,
        minimum height=2em},
    line/.style={draw, -latex'}
}
\tikzcdset{every label/.append style = {font = \small}}

\usepackage[final, backref]{hyperref}

\usepackage{makecell}
\usepackage[novbox]{pdfsync}
\usepackage{amssymb, amsfonts,mathtools,bbm,eurosym,graphicx,latexsym, amsmath,amsthm,times,nccmath,url}
\usepackage{listings,verbatim,cases,tabularx}
\usepackage{graphicx,float,fancybox}
\usepackage{epsfig,epstopdf,tikz}
\usepackage[sans]{dsfont}
\usetikzlibrary{positioning,arrows,arrows.meta,calc}

\newtheorem{theorem}{Theorem}
\newtheorem{definition}{Definition}
\newtheorem{open}{Problem}
\newtheorem{example}{Example}

\newcounter{tabcounter}
\newcounter{defcounter}
\newcounter{opcounter}
\newcounter{thmcounter}
\newcounter{excounter}

\newglossarystyle{unifiedstyle}{
  \setglossarystyle{list}
  
}

\makeatletter
\newcommand{\beXa}[1][]{%
  \def\@temparg{#1}%
  \ifx\@temparg\@empty
    \begin{example}%
  \else
    \refstepcounter{excounter}%  % Step the counter
    \begin{example}[#1]%
    \label{e:#1}%                 % Create a label (using 'e' for example)
    \newglossaryentry{e\theexcounter}{%
      name={#1},%
      description={p.~\pageref{e:#1}},%  % Use pageref of the label
      user1={\theexcounter},%
      user2={Example}%              % Use 'Example' as the category
    }%
  \fi
}
\newcommand{\eeXa}{\end{example}}

\newcommand{\beD}[1][]{%
  \def\@temparg{#1}%
  \ifx\@temparg\@empty
    \begin{definition}%
  \else
    \refstepcounter{defcounter}%
    \begin{definition}[#1]%
    \label{d:#1}%
    \newglossaryentry{d\thedefcounter}{%
      name={#1},%
      description={p.~\pageref{d:#1}},%  % Use pageref of the label
      user1={\thedefcounter},%
      user2={Definition}%
    }%
  \fi
}
\newcommand{\eeD}{\end{definition}}

\newcommand{\beT}[1][]{%
  \def\@temparg{#1}%
  \ifx\@temparg\@empty
    \begin{theorem}%
  \else
    \refstepcounter{thmcounter}%
    \begin{theorem}[#1]%
    \label{t:#1}%
    \newglossaryentry{t\thethmcounter}{%
      name={#1},%
      description={p.~\pageref{t:#1}},%
      user1={\thethmcounter},%
      user2={Theorem}%
    }%
  \fi
}
\newcommand{\eeT}{\end{theorem}}
\newcommand{\beO}[1][]{%
  \def\@temparg{#1}%
  \ifx\@temparg\@empty
    \begin{open}%
  \else
    \refstepcounter{opcounter}%
    \begin{open}[#1]%
    \label{op:#1}%
    \newglossaryentry{o\theopcounter}{%
      name={#1},%
      description={p.~\pageref{op:#1}},%  % Use pageref of the label
      user1={\theopcounter},%
      user2={Problem}%
    }%
  \fi
}
\newcommand{\eeO}{\end{open}}

\newcommand{\betA}[1][]{%
  \def\@temparg{#1}%
  \ifx\@temparg\@empty
    \begin{table}[htbp]%
  \else
    \refstepcounter{tabcounter}%
    \begin{table}[htbp]%

    \label{A:\thetabcounter}%

    \newglossaryentry{A\thetabcounter}{%
      name={#1},% % Uses the Name for the glossary title
      description={p.~\pageref{A:\thetabcounter}},%
      user1={\thetabcounter},%
      user2={Table}%
    }%
  \fi
}

\newcommand{\eetA}{\end{table}}
\makeatother
\newcommand{\printboth}{%
  \glsaddall  % Add all entries to glossary
  \printglossary[style=unifiedstyle, title={Definitions, Theorems, Examples and  Problems}]%
}

\newtheorem{lemma}{Lemma}
\newtheorem{proposition}{Proposition}
\newtheorem{corollary}{Corollary}

\newtheorem{remark}{Remark}
\newtheorem{question}{Question}
\newtheorem{assumption}{Assumption}

\def\beL{\begin{lemma}}\def\eeL{\end{lemma}}
\def\beP{\begin{proposition}}\def\eeP{\end{proposition}}
\def\beC{\begin{corollary}}\def\eeC{\end{corollary}}
\def\beR{\begin{remark}}\def\eeR{\end{remark}}
\def\beQ{\begin{question}}\def\eeQ{\end{question}}
\def\beA{\begin{assumption}}\def\eeA{\end{assumption}}

\definecolor{funccolor}{RGB}{25,25,112}
\definecolor{desccolor}{RGB}{64,64,64}

\def\bep{\begin{pmatrix}}\def\eep{\end{pmatrix}}
\def\bev{\begin{vmatrix}}\def\eev{\end{vmatrix}}
\def\bea{\begin{eqnarray*}}\def\eea{\end{eqnarray*}}
\def\bc{\begin{cases}}\def\ec{\end{cases}}
\def\BEN{\begin{enumerate}}\def\EEN{\end{enumerate}}
\def\BI{\begin{itemize}}\def\EI{\end{itemize}}

\newcommand{\be}[1]{\begin{equation}\label{#1}}
\newcommand{\ee}{\end{equation}}
\newcommand{\beq}{\begin{eqnarray}}
\def\eeq{\end{eqnarray}}
\def\eqr{\eqref}\def\fr{\frac}\def\lbl{\label}
\def\Lra{\Longrightarrow}\def\cNGM{\cite{Diek,Van,Van08}}
\def\Eq{\Leftrightarrow}\def\b{\beta}
\newcommand{\R}{\mathbb{R}}  % Real numbers

\def\al{\alpha}

\def\f{\varphi}  % Kept as requested
\def\ga{\gamma}
\def\La{\Lambda}

\def\T{\widetilde}
\def\H{\hat}

\newcommand{\bff}[1]{{\mbox{\boldmath$#1$}}}

\def\x{\boldsymbol{x}}
\newcommand\y{\boldsymbol{y}}

\def\v1{\vec {\bff 1}}

\def\mR{{\mathcal R}}

\long\def\symbolfootnote[#1]#2{%
\begingroup
\def\thefootnote{\fnsymbol{footnote}}\footnote[#1]{#2}%
\endgroup}

\def\xd{x_{0}}
\def\and{antisymmetric}

\def\Fr{Furthermore, }

\def\ie{i.e. }
\def\im{\item}

\def\nne{non-negative}

\def\para{parameter}

\def\Prf{{\bf  Proof:}}

\def\resp{respectively}
\def\sats{satisfies}\def\saty{satisfy}
\def\sec{\section}\def\ssec{\subsection}

\def\var{variable}

\def\wrt{with respect to}

\newcommand\CRN{chemical reaction networks}

\def\brn{basic reproduction number}
\def\com{compartment}

\def\ME{mathematical epidemiology}

\def\NGM{next generation matrix}

\def\repF{reproduction functions}

\def\fp{fixed point}\def\bfp{boundary fixed point}

\def\RH{Routh-Hurwitz}

\newcommand\BIN{biological interaction network}

\def\LV{Lotka-Volterra}

\def\regS{{\bf regular splitting}}
\def\RUR{rational univariate representation}

 \def\rat{rational substitution}
 \def\RHS{right hand side} \def\LHS{left hand side}
 
\def\wli{$\omega$-limit}
\def\b{\beta}

\newtcolorbox{keyresult}[1][]{
    colback=blue!5!white, colframe=blue!75!black,
    title=#1, fonttitle=\bfseries
}
\newtcolorbox{examplebox}[1][]{
    colback=green!5!white, colframe=green!75!black,
    title=#1, fonttitle=\bfseries
}
\newtcolorbox{proofbox}[1][]{
    colback=yellow!5!white, colframe=orange!75!black,
    title=#1, fonttitle=\bfseries
}

\def\b0{\bm{0}}

\title{
Relay transitions and invasion thresholds in multi-strain rumor models: a chemical reaction network approach
}

\author{Florin Avram$^{1}$ , Rim Adenane$^{2}$,
,Andrei-Dan Halanay$^{3}$
}
\begin{document}

\maketitle
\begin{center}

$^{1}$ Laboratoire de Math\'{e}matiques Appliqu\'{e}es, Universit\'{e} de Pau, Pau, France; avramf3@gmail.com \\
$^{2}$ Laboratoire d'Analyse, Géométrie et Applications, Département des Mathématiques,  Universit\'e Ibn-Tofail, K\'enitra, 14000, Morocco \\
$^{3}$ Department of Mathematics, Bucharest University, Bucharest, Romania

\end{center}

%\iffalse
\begin{abstract}
The historical quest for unifying the concepts and methods of \CRN\ theory
(CRNT), \ME\ (ME) and ecology has received increased attention in the last years
and has led in particular to the development of the symbolic package EpidCRN, for
automatic analysis of positive ODEs, which implements tools from all these disciplines
like siphons, reproduction functions and invasion numbers, Child-Selection expansions, etc.

We illustrate below the convenience of using this package on some recent online social
network (OSN) rumor spreading models, with emphasis on showing how CRNT throws a new
light on their analysis.
Specifically, we organise the boundary dynamics via the lattice of invariant faces
generated by minimal siphons, and establish that stability transitions take the form of
\emph{relays}: for each distance-one cover in the siphon lattice, a single invasion
inequality simultaneously governs the loss of transversal stability of the resident
equilibrium and the existence of a successor equilibrium on the adjacent face.

For the base OSN model ($\omega=0$) all boundary and interior equilibria admit explicit
rational formulas, and the relay table is fully verified using invasion numbers computed
symbolically by EpidCRN.
For the variant with waning spreading impulse ($\omega>0$), the relay structure is
analysed via transversal Jacobian blocks; three equilibria involve irrational coordinates
and their stability is predicted by the relay framework subject to direct Routh--Hurwitz
verification.
The relay mechanism is then situated in its normal-form context (siphon-induced
transcritical bifurcations), distinguished from classical transcritical bifurcations along
four structural axes, and compared with Hofbauer invasion graphs.
\end{abstract}

%\fi
\textbf{Keywords:}    biochemical interaction  networks; essentially nonnegative/positive systems; chemical reaction networks; mathematical epidemiology; disease free equilibrium; regular splitting; \RH\ stability conditions; multi-strain models;   metapopulation models; reproduction functions; admissible communities; invasion numbers;   \CRN; stoichiometric matrix;
  siphons/semi-locking sets;  symbolic Jacobian; rich kinetics;
    Blokhuis-Stadler-Vassena oscillation recipes
   %bifurcations;
   %global asymptotic stability; %stochastic reaction networks; continuous time Markov chains; multi-scale limits; Lyapunov functions %algebraic biology graph theory;  robust control;  optimization approach\\ %, toric steady states, Lotka-Volterra canonical form,

\tableofcontents
\printboth
%
%\newpage
\section{Introduction}
Population dynamics,  ecology,  mathematical epidemiology (ME), virology, rumour spreading, social networks, chemical reaction networks theory (CRNT), are a few of the biological 
interaction networks (BIN) subfields, which study all positive dynamical systems, and have similar preoccupations: the existence and multiplicity of equilibria, their local and
global stability, the occurrence of bifurcations,   persistence, permanence, extinction, etc.

\beD[Positive / non-negative ODEs]
A dynamical system is called \emph{positive} \cite{rantzer2015scalable} or \emph{non-negative} \cite{haddad2010}
if the non-negative orthant
\[
\mathbb{R}^n_{\ge 0} := \{ x \in \mathbb{R}^n : x_i \ge 0,\; i = 1,\dots,n \}
\]
is forward invariant under the flow.
\eeD

\subsection{Motivation}

Positive ODE models arising in chemical reaction networks (CRN), mathematical
epidemiology (ME), ecology, virology,  rumor spreading, etc share a common geometric feature:
their dynamics is constrained by a family of forward--invariant coordinate faces.
 In ME these faces  appear as disease--free or
partially disease--free manifolds, and in  CRN theory they are characterized combinatorially as semi-locking sets, also called \emph{siphons}  in the theory of Petri nets.
 Despite their ubiquity, the systematic role
of invariant faces beyond the disease--free equilibrium (DFE) has remained
fragmented across fields.

The classical next--generation matrix (NGM) theorem
\cNGM, probably the most cited result about positive
ODEs, but  largely unknown outside  (ME),  is emblematic of this fragmentation.
It provides a  striking invasion criterion of the form $R_0 >1$ at the DFE ($R_0$ is called  reproduction number), which is motivated by the theory of population models  
\cite{Lotka} and their branching process  approximations \cite{kendall1966branching}. This result implicitly relies on
three distinct mechanisms that are unfortunately rarely disentangled:

(i) the fact that forward invariance of a boundary face  enforces a {\bf block structure of
the Jacobian  evaluated on it} \cite{AAH26};

(ii) the ensuing {\bf reduction of  stability to that of the  lower--dimensional tangential and transversal blocks} \cite{JA}    ;

(iii) {\bf Metzler structure and regular splittings}. A recurring feature of the models considered in ME is that transversal Jacobian
blocks associated with siphon faces are Metzler matrices.
For such matrices, the equivalence between the sign of the spectral abscissa and
threshold conditions expressed via reproduction numbers goes
back to the theory of regular splittings and non-singular $M$--matrices
\cite{Varga,BermanPlemmons1994}, and
the next--generation matrix construction used in mathematical epidemiology is
just one particular realization of this general principle. At first,
this classic result suggests preferring spectral abscissa, which are unique,
while reproduction numbers aren't. However, in this paper we  use the latter, both due to empirical observations that reproduction numbers are sometimes simpler, and to the probabilistic interpretations which might turn out useful in future works.

A second illustration of the fragmentation of the positive ODE field is the rather rare appearances in \CRN\  theory (CRNT), also known as biological/biochemical interaction networks (BIN) theory,  of the concept of transversal spectral abscissas, which are fundamental in studying multiple boundaries in ecological models  (under the name of  invasion numbers). It seems fair to say that this is caused by CRNT focusing mainly on the interior, while ME focuses mainly on the boundary. The road between the two
are the ``escape from the boundary paths" via positive invasion numbers, originating in  ecology works like
\cite{SchusSig,SigSchus} (also not enough known outside ecology), and  this
context  is the closest to our paper below.

  This paper started as an  investigation of  the escape from the boundary paths of   some multi strain models, an ME topic which may be traced back at least to \cite{ferguson1999effect,schwartz2005chaotic,nuno2005dynamics}, which includes rumor models on online social networks (OSN).  In the process,
  we discovered a striking empirical fact, visible in many concrete models,  that the
inequality governing transversal instability of a resident equilibrium
coincides exactly with the inequality guaranteeing existence of a successor
equilibrium on a strictly larger face. We refer to this phenomenon as a
\emph{boundary transcritical relay}.  While this is reminiscent of transcritical bifurcations,
one aspect is   not included in classical works like \cite{Boldin}, which typically study invasions by one species only.

In positive ODE models however,  equilibria are naturally organized by the lattice induced  by \emph{minimal siphons}, \ie by the
boundary faces generated by unions of minimal siphons, and
 equilibria are invaded by entire \emph{blocks}
of variables corresponding to a minimal siphon, rather than by a single
direction. The relevant invasion object is therefore not a scalar reproduction
number, but the spectral abscissa (or reproduction
number) of a \emph{transversal Jacobian block} associated
with the invading siphon.  \Fr the various reproduction
numbers   in multi-strain models are related between them, being given by plugging equilibrium points in the reproduction functions introduced  in \cite{Blalog,AABH25}.

\paragraph{Reproduction functions as organizing objects.}
One central object of this paper are the  \emph{reproduction functions} defined on siphon faces,  formalized in Definition~\ref{d:reF}.
For simple ME-type siphons, these functions arise as the positive eigenvalues of
transversal Metzler blocks with resident variables left free.
Evaluating the same reproduction function at different resident equilibria
yields, respectively, basic reproduction numbers and invasion numbers.
This viewpoint allows one to  explain why
the same inequality appears both as an existence condition for an invading
equilibrium and as a stability condition for the resident one, and to compare the invasion numbers of the  equilibria across faces, and  therefore partition the parameter space, and  compute bifurcation boundaries.

\paragraph{Relation to invasion graphs in ecology.}
The relay graphs introduced in this paper are closely related in spirit to the
invasion graphs of Hofbauer and collaborators, which encode which communities
can invade which others via positive invasion rates.
There are, however, two structural differences.
First, our nodes are not arbitrary equilibria but equilibria indexed by
\emph{inhabited siphon faces}, ordered by inclusion in the lattice generated by
minimal siphons.
Second, edges in our relay graph are not abstract invasion relations but are
labeled by explicit invasion numbers obtained as evaluations of reproduction
functions.
In particular, adjacency in the relay graph corresponds to distance--one moves
in the siphon lattice, which allows an algorithmic enumeration of equilibria
and a direct correspondence between instability of a resident equilibrium and
existence of its successor.

The first goal of this paper is to further explain, formalize, and algorithmize the boundary transcritical relay (BTR)
mechanism.
We claim  that the coincidence between instability of a resident
equilibrium and existence of an invading equilibrium is not accidental, nor a
model--specific algebraic trick. Rather, it is a structural consequence of:
(a) face invariance induced by siphons;
(b) factorization of the vector field in the transversal variables;
(c) Metzler structure of the corresponding Jacobian blocks.
When these conditions hold, invasion by a minimal siphon necessarily produces a
boundary transcritical bifurcation on the invariant face, and the relay follows.

To render this structure visible, we advocate a shift of viewpoint:
from individual equilibria to the \emph{lattice of inhabited siphon faces}.
Within this lattice, equilibria are naturally ordered, and the global boundary
dynamics is organized as a directed relay graph whose edges are labeled by
invasion numbers.
This perspective unifies tools from CRN theory (siphons, semilocking sets),
ME (reproduction numbers  and functions), and ecology (escape from the boundary paths,invasion graphs)
and it leads to concrete algorithms for enumerating equilibria, testing their
stability, and predicting which equilibria replace which others as parameters
vary.

The OSN rumor model studied here provides a nontrivial but tractable test case.
It exhibits multiple minimal siphons, non-rational equilibria, and competing
relay chains.
By analyzing first the case $\omega=0$, where the relay structure is completely
transparent, and then extending the analysis to $\omega>0$ %via perturbation and
%rank--one feedback arguments,
we show how the relay mechanism persists even when
explicit equilibrium formulas are unavailable.
\iffalse
Our quest to understand NGM and invasions led to the development of the symbolic package  EpidCRN, for  automatic analysis of positive ODEs,  which implements  several CRNT, ecology  and ME tools (see below). %like siphons, boundary searches for equilibria, reproduction functions,  Child-Selection  expansions, etc.

We  argue below that a unification of BIN subfields is profitable both from pedagogical and computational points of view.
\BEN \im
Pedagogically,
 since several crucial BIN facts which we call sometimes ``laws", %in the apendix,
 that have emerged in different subfields of BIN, and algorithms, are both not well-enough known outside the field
where they originated, and not enough understood.

\im Computationally, we    illustrate  how different flavor examples  may be   studied
in a unified and off-shelf manner, just by entering the ODE, and then calling the \texttt{Epid-CRN} Mathematica package at \url{https://github.com/florinav/EpidCRNmodels}. This contains  modules to identify minimal siphons, rational quasi-steady state solutions (QSS),  rational boundary fixed points, F,V regular splitting matrices, the \NGM\ matrix K, etc,  integrating thus
methods for positive ODEs that have emerged in different subfields like \ME\ (ME), \CRN\ (CRN) and ecology.
\EEN
\fi
\subsection{Structure of the paper}

The paper is organized around the idea that the global boundary dynamics of a
positive ODE model is governed by the lattice of invariant faces generated by
minimal siphons, and that stability transitions occur as relays between equilibria
living on adjacent faces.

Section~\ref{s:OSN} introduces the formal framework, and exemplifies it
via the OSN model  of Fakih, Halanay  and Avram  \cite{HA}.  We also illustrate how EpidCRN facilitates its analysis, by interpreting  it as a chemical reaction network, which
identifies its minimal siphons and the lattice they generate.
Using the fact that face invariance enforces a block structure of the Jacobian
\cite{AAH26},  we
introduce the notion of a transversal Jacobian block associated with a siphon.
This section provides the structural objects---faces, blocks, and invasion numbers---
that are used throughout the paper.

Section~3 analyzes in detail the case $\omega=0$, where all equilibria admit explicit
formulas.
This section constitutes the conceptual core of the paper.
We enumerate all face equilibria, organize them along the siphon lattice, and
establish the relay table linking existence and stability.
The main result shows that, for each distance--one step in the lattice, the same
invasion inequality governs both the loss of transversal stability of a resident
equilibrium and the emergence of its successor on a less constrained face.
This makes the relay mechanism fully transparent and algorithmic.

Section~\ref{s:alg} extends the relay viewpoint in two directions.
First, it formulates an algorithmic test for relays along distance--one covers in the
siphon lattice: given a cover pair $(T,S)$ the algorithm checks (i) existence and
tangential stability of a face equilibrium $E_S$, (ii) transversal instability via the
Metzler block $M_\sigma(E_S)$, and (iii) existence and stability of the successor $E_T$;
termination is guaranteed whenever all intermediate characteristic polynomials are
rational in the parameters.
Second, the section situates transcritical relays in their normal-form context,
identifies four structural differences from classical transcritical bifurcations
(geometric origin of the critical eigenvalue, block invasion rather than single direction,
lattice-prescribed successor, shared inequality for instability and successor existence),
and compares the relay-graph framework with Hofbauer invasion graphs.

Section~\ref{s:omp} provides the application of the relay algorithm to the case $\omega>0$.

Section~\ref{s:Hopf}  provides an auxiliary structural result, namely the impossibility of oscillatory
behavior for the case $\omega=0$,
which clarifies the dynamical landscape, but is not required for the main
relay analysis.

Section~\ref{s:R1}  discusses perturbative extensions to the case
$\omega>0$, where explicit equilibria are no longer available.
Here we show how rank--one feedback and resolvent estimates can be used to control
stability of candidate relay equilibria identified by the siphon analysis.
This section is exploratory in nature and is included to indicate how the relay
framework can be combined with perturbation theory in more general settings.

Appendix~\ref{s:NGM} includes, for convenience of the reader,  our generalization of
 the NGM theorem from \cite{AAH26}.   

 Finally, Appendix~\ref{s:wl} recalls, for convenience of the reader, 
 two crucial  results of
 CRNT: the equivalence between invariant boundary faces and siphons, and the fact
 that boundary $\omega$-limit point $E$ must reside on (critical)  siphon faces~
\cite{AdLS}, which provide the foundation of our approach.

Sections~3--6 can be read as a progression of decreasing explicit information:
rational equilibria (Section~3), algorithmically detected relays for general positive
ODEs (Section~\ref{s:alg}), implicit equilibria with relay predictions intact
(Section~\ref{s:omp}), and perturbative stability control when neither equilibria nor
their stability are explicit (Section~\ref{s:R1}).
%Throughout, the relay mechanism itself---organized by the siphon lattice and invasion numbers---remains unchanged.

\iffalse

We conclude with a discussion of implications for multi--strain epidemic models,
rumor spreading on networks, and CRN theory, emphasizing how invasion numbers at
non--DFE equilibria naturally organize the global boundary dynamics.

  Section \ref{s:CS} reviews the main ideas of  the Vassena-Stadler bifurcation analysis of symbolic Jacobians via Cauchy--Binet Child Selections,  and illustrates it on  the SIRWS ME model.

   In section \ref{s:BR}, we review and comment  the recent idea of \cite{BorRost}  of
   proving symbolically the presence of Hopf bifurcations for complicated ME models by examining reduced subnetworks and by applying ``inheritance rules" \cite{Banaji,BBH25}.

 Section \ref{s:CRW} reviews results of \cite{VAA}
  on a general  ``Capasso-Ruan-Wang" ``\spf" SIRS-type epidemic model with nonlinear forces of infection and treatment \eqr{Vyss} (which generalizes   Zhou-Fan).
\fi

\sec{A  model of propagation of rumors in  online social networks (OSN)   of Fakih, Halanay  and Avram  \cite{HA}, with 8 species  13 \para s, and 6 equilibria: OSN.wl
}
\lbl{s:OSN}

This model is brought here because it exhibits simultaneously several features common to positive ODE
models: multiple minimal siphons, their lattice, non-rational equilibria, non mass-action kinetics and competing invasion
paths. In particular, it lead us to develop a script clsEq, which, using  minimal siphons and an associated  polynomial complementarity principle, determines   the equilibria, and   classifies them into A) rational, B) admitting a \RUR\ (RUR), or C) undecidable under given time constraints;
It     also lead us to introduce  fundamental  concepts like ME-type siphon-faces, defined by  the Jacobian $J_x(\y)$  being Metzler, and    admitting some regular splitting  $J_x=F-V$  --see below.

The similarity between viruses and rumours is well known, since the pioneering paper \cite{DalKen}. For further work
the  propagation of rumors in  online social networks (OSN),  see \cite{amaral2018rumor,zehmakan2019fake,li2021stochastic,ghosh2023comparative,
wenkai2021taming,
WangOSN,zhou2023stability,yin2024dynamics,HA,FH}.

The model studied in \cite{HA}, defined by the ODE
\be{AOSN}X'=\bep x_1\\ U\\S_1\\S_2\\ B_1\\B_2\\R\\ W\eep'=\left(
\begin{array}{c}
 \La -\mu x_1-\beta x_1 U \\
 -U \left(\mu -\beta  x_1+\beta_w W\right) \\
 \frac{\beta _1 B_1 U}{B_1 \epsilon _1+\alpha _1 U+1}-\gamma _1 S_1 \\
 \frac{\beta _2 B_2 U}{B_2 \epsilon _2+\alpha _2 U+1}-\gamma _2 S_2 \\
 \gamma _1 S_1-B_1 \mu _1 \\
 \gamma _2 S_2-B_2 \mu _2 \\
 B_1 \mu _1+B_2 \mu _2-\omega R\\
 \beta_w U W-\mu  W+R\omega
\end{array}
\right),\ee
contains mostly standard reactions, with  the exception of the two inflows with (increasing) fractional rates.

 The model tracks  8 state variables:
\begin{itemize}
\item $x_1=P$: are  \textbf{potential users of OSN}, who could join it, but have not yet done so.
\item $U$: \textbf{New users of OSN},  not yet aware of the rumors. This is a \LV\ species.

\item $S_1$: \textbf{Users susceptible for rumor 1} (they heard rumor 1 and are thinking about it).
\item $S_2$: \textbf{Users susceptible for rumor 2}.
\item $B_1$: \textbf{Believers and spreaders of rumor 1} (those actively spreading or convinced by rumor 1).
\item $B_2$: \textbf{Believers and spreaders of rumor 2}.
\item $R$: \textbf{Users who are now skeptic of both rumors} (OSN users who believed, but do not believe any more either rumor).
    \item $W$: \textbf{Ex/withdrawn -users} of OSN, who have abandoned
the platform, and may moreover convince other users to abandon.
\end{itemize}

\begin{figure}[H]
  \centering
  \begin{tikzpicture}[
      scale=0.95, transform shape,
      node distance=1.8cm and 2cm,
      state/.style={circle, draw, minimum size=1cm, font=\large},
      compartment/.style={rectangle, draw, minimum width=1cm, minimum height=0.7cm, font=\normalsize},
      arrow/.style={-Stealth, thick},
      label/.style={font=\scriptsize, align=center}
  ]

  % Main horizontal flow: x1 -> U -> x3
  \node[state] (x1) {$x_1$};
  \node[state, right=2.5cm of x1] (U) {$U$};
  \node[state, right=3cm of U] (x3) {$W$};

  % R node precisely below U
  \node[compartment, below=3.5cm of U] (R) {$R$};

  % Left branch pointing DOWN: S1 -> B1
  \node[compartment, left=1.5cm of R] (B1) {$B_1$};
  \node[compartment, above=1.2cm of B1] (S1) {$S_1$};

  % Right branch pointing DOWN: S2 -> B2
  \node[compartment, right=1.5cm of R] (B2) {$B_2$};
  \node[compartment, above=1.2cm of B2] (S2) {$S_2$};

  % Influx to x1
  \node[left=0.8cm of x1] (source) {};
  \draw[arrow] (source) -- node[above, label] {$\Lambda$} (x1);

  % x1 -> U: platform adoption
  \draw[arrow] (x1) -- node[above, label] {$\beta x_1 U$} (U);

  % U -> S1: rumor 1 creation
  \draw[arrow, dashed] (U) -- (S1);

  % U -> S2: rumor 2 creation
  \draw[arrow, dashed] (U) -- (S2);

  % S1 -> B1: rumor 1 conversion (pointing down)
  \draw[arrow] (S1) -- node[left, label] {$\gamma_1 S_1$} (B1);

  % S2 -> B2: rumor 2 conversion (pointing down)
  \draw[arrow] (S2) -- node[right, label] {$\gamma_2 S_2$} (B2);

  % B1 -> R
  \draw[arrow] (B1) -- node[below, label] {$\mu_1 B_1$} (R);

  % B2 -> R
  \draw[arrow] (B2) -- node[below, label] {$\mu_2 B_2$} (R);

  % R -> x3
  \draw[arrow] (R) -- node[left, label, pos=0.4] {$\omega R$} (x3);

  % U -> x3: platform abandonment (curved)
  \draw[arrow] (U) to[bend left=35] node[above, label] {$\beta_w U W$} (x3);

   % Horizontal inflows to S1 and S2 with fractional rates
    \node[left=1.2cm of S1] (srcS1) {};
    \draw[arrow] (srcS1) -- node[above, label] {$\frac{\beta_1 B_1 U}{B_1\epsilon_1 + \alpha_1 U + 1}$} (S1);

    \node[right=1.2cm of S2] (srcS2) {};
    \draw[arrow] (srcS2) -- node[above, label] {$\frac{\beta_2 B_2 U}{B_2\epsilon_2 + \alpha_2 U + 1}$} (S2);

  % LEGEND to the right of x3
  \node[draw, fill=yellow!10, rounded corners, align=left, font=\scriptsize,
        anchor=north west] at (8.5, 0.5) {
  \begin{tabular}{@{}l@{\quad}l@{}}
  \multicolumn{2}{c}{\textbf{Reactions \& Rates}} \\[2pt]
  \hline \\[-6pt]
  $0 \to x_1$ & $\Lambda$ \\[1pt]
  $x_1 + U \to 2U$ & $\beta x_1 U$ \\[1pt]
  $U + W \to 2W$ & $\beta_w U W$ \\[1pt]
  $B_i + U \to B_i + U + S_i$ & $\frac{\beta_i B_i U}{B_i\epsilon_i + \alpha_i U + 1}$ \\[6pt]
  $S_1 \to B_1$ & $\gamma_1 S_1$ \\[1pt]
  $S_2 \to B_2$ & $\gamma_2 S_2$ \\[1pt]
  $B_1 \to R$ & $\mu_1 B_1$ \\[1pt]
  $B_2 \to R$ & $\mu_2 B_2$ \\[1pt]
  $R \to W$ & $\omega R$ \\[1pt]
  $x_1 \to 0$ & $\mu  x_1$ \\[1pt]
  $U \to 0$ & $\mu  U$ \\[1pt]
  $W \to 0$ & $\mu  W$ \\
  \end{tabular}
  };

  \end{tikzpicture}
\caption{A flow-diagram showing the compartmental modeling of  two rumors,  1 and  2, spreading among the users of a social media platform. Each of the 13  rates of the model is  associated
to an edge, but the three ``decay edges  "   for $x_1,U,W$ are omitted in the diagram.  8 of the edges   correspond to mass action  reactions (the   five lower ``conversion transitions" with linear rates, and the quadratic ``epidemic SEIR" transitions), which allow reconstructing parts of the ODE. The  two fractional, non mass-action rates are  actually flows coming from outside, catalyzed by $U$ (which was indicated by the dashed arrows). Note that if $\omega=0$, then $R$ does not affect the rest of the species, and may be dropped.}
\label{f:OSN}
  \end{figure}
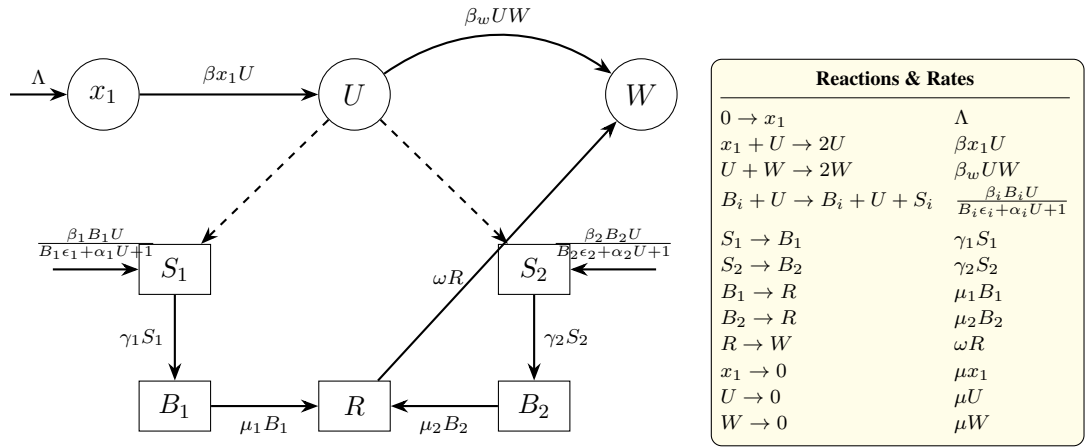

   The next sections provide  comments on the associated file OSN.wl, which illustrate that currently we have a promising automatization of stability analysis in our package. The user inputs
 a triple $$(RHS,var,prF),$$
  where the first two represent the model, and prF is a ``pretty print format" which makes the results look nice in .tex. The package converts first to standard CRN representation (RN,rts), and accomplishes then routine tasks like determining the siphons,   classifying the  fixed points as rational, admitting a \RUR\ (RUR), or timing out, etc.

\ssec{The script ODE2RNp[RHS,var,prF] gets  a ``perturbed" mass-action decomposition  (RN,rts) of the ODE
}

\beR To convert an (RHS,var) to the (RN, rts) representation customary in CRNT, we extract first the list of symbolic monomials rts which are contained in  RHS, each included only once, up to  proportionality (so, $\al x y, \al \ga x y$ will generate one rate only, $\al x y$, or $ \al \ga x y,$ depending on which appears first). At this point, we may write
$$RHS=gam.rts,$$ where $gam$ is a constant matrix. Next, we decompose
$gam=bet-alp,$ where $bet,alp$  are \nne\ matrices which will be used to define the coefficients of the \RHS\ and \LHS\ of the reactions. \eeR

\beR \lbl{r:P} \BEN \im The last 8 ``mass action linear/conversion" reactions in Figure \ref{f:OSN} require no explanation.
\im Mass action rule imposes writing the second, platform adoption  reaction,
with rate $\beta  x_1 U$   as
 $$\text{x1}+U\to 2 U,\quad \text{instead of} \quad  \text{x1}\to  U, \text{ as used in ME}.$$

 Note also the mass action writing explains the mechanism behind the phenomenon: a ``non-user"+ ``user" interaction yields two users.
 \im Similarly, the third,  abandonment  reaction, is mass action.
\im The fourth reaction in Figure \ref{f:OSN} with rate
$$f_j(U,B_j)=\frac{\beta _j B_j U}{B_j \epsilon _j+\alpha _j U+1}, j=1,2$$ is a  ``catalyzed rumor j creation"
 $$\text{Bj}+U\to \text{Bj}+U+\text{Sj}, \quad (\text{instead of  }\quad 0\to \text{Sj});$$  all the variables of the rate have been by added on both sides, since they need to appear in the LHS.
 %Now,   after neglecting  the denominator, the reaction is mass action. % (we will call rates which are polynomial divided by a denominator $1+...$ perturbed mass action rates).

 \im The transformation of $R$ into $W$ is linear, suggesting individual decisions, while the transformation of $U$ into $W$ is quadratic, ``mass action" type,  suggesting decisions based on contacts with other ex-users.

In fact, both these transformations might occur both linearly and mass-action,  suggesting more complicated models, but here we will study exactly
the \cite{HA} model, while providing some CRN perspective.

\im The division of \com s in rumor related ones $S1,S2,B1,B2,R$, involved in fast transitions, where  decay parameters have been neglected for simplicity, and population \com s $x_1, U, W$, subject to decay.

\EEN
\eeR

\beR Note that  the RN representation is not unique, and that  the search for ``good representations" is a very active field -- see for example \cite{Buxton,Hong} (whether the ``good representations" in chemistry, which enjoy weak-reversibility and low-deficiency, for example, will also be useful in ME is a topic for further research).  The main theoretical CRN concept used in this paper, the set of minimal siphons, depends only on the ODE and not on the RN representation.\eeR

\ssec{The script minSiph[var,RN] gets   the minimal siphons
}
The minimal siphons  $$\{U\},\{B1,S1\},\{B2,S2\}$$
have clear meanings (absence of OSN users, of rumor 1 and of rumor 2, \resp).

The  most relevant siphons for us are the
``DFE siphon" $\{U,B1,S1,B2,S2\}$ and the ``inhabited siphon faces"
\beD[DFE siphon, lattice of minimal siphons, inhabited siphons]\lbl{d:ins}
A) The DFE siphon is the union of the minimal siphons, and its face is the intersection of all invariant faces.

B) An inhabited siphon  is a siphon belonging to the lattice generated by the minimal ones, and  whose face contains $\omega-$limit points which are not present in larger siphons of the lattice (equivalently, they belong to the { relative interior},  if all the non-DFE variables are deleted).
\eeD

There are also 13 non minimal siphons here, but our package needs only five of them, those which complete the lattice generated by the minimal siphons. \Fr here only the rumour free (RF) siphon
$\{B1,S1,B2,S2\}$ is ``inhabited".

\beR Note that the U siphon is  caused by a variable which factors out in its equation.
\eeR
\beD[Lotka-Volterra siphons]
Siphons caused by variables which factor out in their equation will be called \LV\ siphons.
\eeD

\beR The ODE \eqr{AOSN} is interesting, since despite its fixed point system
 being  hard enough so that a brute force
Mathematica approach via Solve  does not succeed  on our computer,  it was solved by  hand in \cite{HA}, and lead to developing the script clsEq. This classifies the equilibria in rational, non rational with explicit RUR, and time-outs, during a hierarchical search of the lattice of inhabited siphons, starting from the DFE  downwards, and  using the complementarity principle well-known in the theory of \LV\ models (essentially, we replace  the LV species by  minimal siphons).

An alternative would be to search over all boundary faces, or over all siphons,
but there is reasonable hope that our ``CRN informed algorithm" performs better.
\eeR

The script clsEq finds here 5 {\bf inhabited lattice siphons} -- see Definition \ref{d:ins}.
As in all good models, they  have clear interpretations:

\BEN

\im $E_2$ (Only rumor 2), with $S_1=B_1=0$.

\im $E_1$ (Only rumor 1), with $S_2=B_2=0$.

\im $RF$ (Rumor-free), with $S_1=S_2=B_1=B_2=0$.

\im $OSND$ (OSN disappearance)  with $S_1=S_2=B_1=B_2=U=0$.

\im $EE$ (endemic equilibria), with all species positive.

\EEN

We will find that while the OSND, and two fixed points found on RF, gOSN and RFE are rational,
 E1, E2  are not, but that quadratic polynomial \RUR s (RUR) are available, in the ``keep variable" $x_1$ (or  $U$), and
 the EE (endemic equilibrium with all species positive) also \sats\ a  quadratic polynomial equation in $x_1$.
 The nonrationality complicates the analysis, and motivated us to start with the case
 $\omega=0 $  in Section \ref{s:om}.% but in current EpidCRN version, this times out}.
 
 Next section illustrates how the siphon structure reflects in lower triangular structures of the Jacobian, cf.  section \ref{s:NGM}.
\subsection{The symbolic Jacobian and its block structure}\lbl{s:Jac}

To illustrate the block structure,  we chose an order 
$(S_1,B_1,S_2,B_2,U,R,W,x_1)$  which starts with the DFE \var s. The symbolic Jacobian is sparse, without block structure (but is  a ``rank-one perturbation"  of a lower  triangular matrix -- see Section \ref{s:R1}):

\be{sJ}sJ=\left(
\begin{array}{cc|cc|c||ccc}
 -\gamma _1 & f_{1,B} & 0 & 0 & f_{1,U} & 0 & 0 & 0 \\
 \gamma _1 & -\mu _1 & 0 & 0 & 0 & 0 & 0 & 0 \\
 \hline 
 0 & 0 & -\gamma _2 & f_{2,B} & f_{2,U} & 0 & 0 & 0 \\
 0 & 0 & \gamma _2 & -\mu _2 & 0 & 0 & 0 & 0 \\
 \hline
 0 & 0 & 0 & 0 & -\mu _{\text{n}}-W \beta _w+\beta  x_1 & 0 & -U \beta _w & \beta  U \\
 \hline
 \hline
 0 & \mu _1 & 0 & \mu _2 & 0 & -\omega  & 0 & 0 \\
 0 & 0 & 0 & 0 & W \beta _w & \omega  & U \beta _w-\mu  & 0 \\
 0 & 0 & 0 & 0 & -\beta  x_1 & 0 & 0 & -\mu -\beta  U \\
\end{array}
\right)\ee

When $U=0$, the only resident variable is $x_1$, and sJ becomes
\bea sJU=\left(
\begin{array}{cccccccc}
 -\gamma _1 & 0 & 0 & 0 & 0 & 0 & 0 & 0 \\
 \gamma _1 & -\mu _1 & 0 & 0 & 0 & 0 & 0 & 0 \\
 0 & 0 & -\gamma _2 & 0 & 0 & 0 & 0 & 0 \\
 0 & 0 & \gamma _2 & -\mu _2 & 0 & 0 & 0 & 0 \\
 0 & 0 & 0 & 0 & \beta  x_1-\mu _{\text{n}} & 0 & 0 & 0 \\
 0 & \mu _1 & 0 & \mu _2 & 0 & -\omega  & 0 & 0 \\
 0 & 0 & 0 & 0 & 0 & \omega  & -\mu  & 0 \\
 0 & 0 & 0 & 0 & -\beta  x_1 & 0 & 0 & -\mu  \\
\end{array}
\right),\eea
so that when  $   x_1>\fr{\mu _{n}}{\beta}:=\H x_1$, instability holds.

In the interior of the RFE siphon, where  $S_1=B_1=S_2=B_2=0, U>0$, sJ becomes
\bea sJ=\left(
\begin{array}{cc|cc|cccc}
 -\gamma _1 & f_{1,B} & 0 & 0 & 0 & 0 & 0 & 0 \\
 \gamma _1 & -\mu _1 & 0 & 0 & 0 & 0 & 0 & 0 \\
 \hline
 0 & 0 & -\gamma _2 & f_{2,B} & 0 & 0 & 0 & 0 \\
 0 & 0 & \gamma _2 & -\mu _2 & 0 & 0 & 0 & 0 \\
 \hline
 0 & 0 & 0 & 0 & -\mu _{\text{n}}-W \beta _w+\beta  x_1 & 0 & -U \beta _w & \beta  U \\
 0 & \mu _1 & 0 & \mu _2 & 0 & -\omega  & 0 & 0 \\
 0 & 0 & 0 & 0 & W \beta _w & \omega  & U \beta _w-\mu  & 0 \\
 0 & 0 & 0 & 0 & -\beta  x_1 & 0 & 0 & -\mu -\beta  U \\
\end{array}
\right)\eea
confirming the siphon structure on the RFE face, with three diagonal blocks of sizes (2,2,4); this confirms also the siphon structure on the $E_2,E_1$ faces.

\ssec{The scripts NGM[RHS, var,sip] and NGMRN[RN, rts, var] find (F,V,K), \repF\ and resident equilibria for ME-type siphons  (admitting \regS s)}

\beD [ME-type siphon]\label{d:VdD}
Let $\x$ denote the \var s which are $0$ on a siphon face S, and let  $\y$ denote the positive (resident) \var s. We will call ME-type siphon a siphon which \sats\
that the Jacobian $J_x(\y)$ on S is Metzler (equivalently, the ODE is locally  increasing in $\y$ on S), and    admits some regular splitting  $J_x=F-V$
(recall that the matrix $K=F V^{-1}$ is called NGM).
\eeD

\beR While spectral radia of NGM matrices may be non-rational, for the multi-strain models in this paper, which do not involve both hosts and vectors,   we only find rational ones. \eeR

\beD [simple ME-type siphon]\label{d:simpM}
An ME-type siphon for which the NGM  has only zero eigenvalues and rational,  unconditionally positive eigenvalues, which  were called
\repF\ in \cite{AABH25}, will be called simple.

\eeD

\beXa[DFE siphon and its reproduction function]
The maximum lattice siphon $\{U,B1,S1,B2,S2\}$ hosts  the
\textbf{online social network disappearance equilibrium} (OSND)
$$
E_0 = \left(\xd:=\fr{\La}{\mu},\,0,\,0,\,0,\,0,\,0,\,0,\,0\right).
$$

 By NGM or direct approach, we find it is stable
  when the \textbf{ reproduction function} \cite{AABJ,AABH25}
  $$ R(\y)=\frac{\beta  x_1}{\mu _n+\beta_w W}$$
  \sats:
  \be{R0r} 1> R_0:=R(E_0)=\frac{\xd }{\H x_1}, \xd:
 =\fr{\La}{\mu}, \H x_1:=\frac{\mu _n}{\beta } \quad   \text{ (see also \eqr{gOs})}
\ee
\eeXa
  The motivation for introducing reproduction functions is that they are involved in the
 computation of invasion numbers of other boundary equilibria, besides the DFE, and that
  they fit our philosophy of switching focus from the \bfp s to the faces which host them.

\beD[reproduction functions, invasion numbers, and reproduction  numbers for simple siphon faces] \label{d:reF}
\leavevmode

For a simple ME-siphon face:
%\noindent
\BEN \im the unique strictly positive eigenvalues of the  blocks, $R_i(\y), i=1, ...k$, with the $\y$ \var s left free, will be called {\bf R-reproduction functions}
(associated with the  block $i$).

\im  For any \fp\ E in this siphon face
\be{Ri} R_i(E), i=1,...,k\ee
will be called the  {\bf \brn\ of block} $i$.

\im   In the case $k=2$
with two blocks  which correspond to two sub- siphons with resident \fp s
$E_j$,
\be{invN}R_i^{\T j}:=R_i(E_j)\ee will be called
the {\bf invasion number of invading block $i$ on resident block $j$}.
%(see below for the \gene\ to multi-strain).

\EEN
\eeD

\beXa [an inhabited ME-type siphon with two reproduction functions] The RF siphon with zero set $S_1,S_2,B_1,B_2 $
has NGM
$$
\left(
\begin{array}{cccc}
 0 & 0 & 0 & 0 \\
 \frac{\beta _1 U}{\mu _1 \left(\alpha _1 U+1\right)} & \frac{\beta _1 U}{\mu _1 \left(\alpha _1 U+1\right)} & 0 & 0 \\
 0 & 0 & 0 & 0 \\
 0 & 0 & \frac{\beta _2 U}{\mu _2 \left(\alpha _2 U+1\right)} & \frac{\beta _2 U}{\mu _2 \left(\alpha _2 U+1\right)} \\
\end{array}
\right)
$$

The \repF\ are
$$R_1(U)=\frac{\beta _1 U}{\mu_1(1+\alpha _1  U)},R_2(U)=\frac{\beta _2 U}{\mu_2(1+\alpha _2  U)}.$$

The RF siphon hosts  two rational equilibria:
 \BEN \im gOSN: `` good OSN", which besides being rumor free is also  without retired ex-users, \ie\ $\{W=R=0\}$.  The non-zero coordinates are
\be{gOs} \H x_1= \frac{\mu _n}{\beta }:=\fr 1{\mR},\H U=
\frac{\mu }{\beta }(R_0-1)=\frac{\mu  }{\mu _n}(\xd-\H x_1), \H x_1<\xd.\ee

\im RFE (rumor free equilibrium OSN) whose zero coordinate set includes also $\{R=0\},$ and the non-zero coordinates are \be{RFE}\T x_1= \frac{\beta_w \Lambda }{\mu  (\beta  +\beta_w)}=\xd \frac{\beta_w }{\beta  +\beta_w},\T U= \frac{\mu }{\beta_w },
\T W= \frac{\beta  }{\beta_w}(\T x_1 -\H x_1), \H x_1<\T x_1 \Eq R_0 >1 + \fr {\beta}{\beta_w}.\ee

\EEN
\eeXa

Since the OSN model \eqr{AOSN} is pretty  challenging, we will consider in next section first the particular case $\omega=0 $, and come back to $\omega>0 $ in section \ref{s:omp}. 
\section{The case $\omega=0 $}\lbl{s:om}
When $\omega=0$ the $R$~compartment decouples from the OSN dynamics (recovered individuals no longer re-enter the susceptible pool via the withdrawal route), reducing the model to $7$~dimensions.  We have one more \LV\ siphon, W, and $9$  rational solutions.

 \noindent\scalebox{0.95}
 {$\bc
  (x_1=x_0= \frac{\Lambda }{\mu },U= 0,W= 0,S_1= 0,S_2= 0,B_1= 0,B_2= 0) &{\quad\mbox{(DFE)}}\\
  \hline
  \hline
 (\H  x_1= \frac{\mu _{\text{n}}}{\beta },\H U= \frac{\Lambda }{\mu _{\text{n}}}-\frac{\mu }{\beta },W= 0,S_1= 0,S_2= 0,B_1=
  0,B_2= 0)&{\quad\mbox{(gOSN)}}
  \\
 (\H x_1= \frac{\mu _{\text{n}}}{\beta },\H U= \frac{\Lambda }{\mu _{\text{n}}}-\frac{\mu }{\beta },W= 0,
\H S_1= \frac{\beta _1 \H U -
  \mu _1 \left(1+ \alpha _1 \H U\right)}{ \gamma _1 \epsilon _1 },S_2= 0,\H B_1= \frac{\beta _1 \H U -
  \mu _1 \left(1+ \alpha _1 \H U\right)}{\epsilon _1
  \mu _1 },B_2= 0)&{\quad\mbox{(E1g)}}
  \\
  (\H x_1= \frac{\mu _{\text{n}}}{\beta },\H U= \frac{\Lambda }{\mu _{\text{n}}}-\frac{\mu }{\beta },W= 0,S_1= 0,\H S_2= \frac{\beta _2 \H U -
  \mu _2 \left(1+ \alpha _2 \H U\right)}{\gamma _2 \epsilon _2 },B_1= 0,\H B_2= \frac{\beta _2 \H U -
  \mu _2 \left(1+ \alpha _2 \H U\right)}{\epsilon _2
  \mu _2 }) &{\quad\mbox{(E2g)}}
  \\
  (\H x_1,\H U,W= 0,\H S_1,\H S_2,\H B_1,\H B_2) &{\quad\mbox{(EEg)}}\\
  \hline
  \hline
  (\T x_1= \frac{\Lambda \beta _w}{\mu \left(\beta +\beta _w\right)},\T U= \frac{\mu }{\beta _w},\T W= \frac{\beta \Lambda }{\mu
  \left(\beta +\beta _w\right)}-\frac{\mu _{\text{n}}}{\beta _w},S_1= 0,S_2= 0,B_1= 0,B_2= 0) &{\quad\mbox{(RFE)}}
  \\
  (\T x_1= \frac{\Lambda \beta _w}{\mu \left(\beta +\beta _w\right)},\T U= \frac{\mu }{\beta _w},\T W= \frac{\beta \Lambda }{\mu
  \left(\beta +\beta _w\right)}-\frac{\mu _{\text{n}}}{\beta _w},\T S_1= \frac{R_1(\T U)-1}{\beta _w \gamma _1 \epsilon _1},S_2= 0,\T B_1= \frac{R_1(\T U)-1}{\beta _w \epsilon
   _1 \mu _1},B_2= 0)&{\quad\mbox{(E1)}}
  \\
 (\T x_1= \frac{\Lambda \beta _w}{\mu \left(\beta +\beta _w\right)},\T U= \frac{\mu }{\beta _w},\T W= \frac{\beta \Lambda }{\mu
  \left(\beta +\beta _w\right)}-\frac{\mu _{\text{n}}}{\beta _w},S_1= 0,\T S_2= \frac{R_2(\T U)-1}{\beta _w \gamma _2 \epsilon _2},B_1= 0,\T B_2= \frac{R_2(\T U)-1}{\beta _w \epsilon _2 \mu _2})&{\quad\mbox{(E2)}}
  \\
 (\T x_1,\T U,\T W,\T S_1,\T S_2,\T B_1,\T B_2)&{\quad\mbox{(EE)}}
   \ec
   $
}

  \beR The above computations which involve  15 non-zero resident values, all expressible as evaluations of  \repF, suggest  the following notation scheme.  For each maximal non-DFE inhabited siphon, we need to evaluate the
  reproduction function  at all the resident points. In this example there is only one, with two resident points, so we need to introduce
  $$R1r=R_1(\T U), R2r=R_2(\T U),R1g=R_1(\H U),R2g=R_2(\H U).$$

  The final partition of the parameter space requires ordering these 4 Rs, $R0=R_0(x_0),$ and $1$ in all the possible ways.\eeR

%%%%%%%%%%%%%%%%%%%%%%%
\ssec{Capturing global boundary behavior via the boundary transcritical relay graph when $\omega=0$}
\lbl{s:relO}

The eight boundary \fp s  naturally split into two parallel families of $4$,
according to whether the withdrawal-siphon reproduction number,
given  by the reproduction function
of the new siphon $W$, evaluated at the gOSN equilibrium value $\hat U$ of $U$,
 is
below $1$, or above:
\[
  R_0^W \;=\; \frac{\beta_w \hat{U}}{\mu} \;=\; \frac{\beta_w}{\beta}(R_0-1)<1 \Eq
R_0 <1+\frac{\beta}{\beta_w}\]

Here $\hat U = \Lambda/\mu_n - \mu/\beta = \mu(R_0-1)/\beta$ is derived from the
gOSN balance $\beta \hat x_1 = \mu_n$, $\mu_n \hat U = \Lambda - \mu\hat x_1$.
\iffalse
\begin{itemize}
  \item \textbf{DFE} (disease-free): stable iff $R_0<1$; exists always.
  \item \textbf{gOSN} (OSN-only, no strains, no withdrawal route): stable iff $R_0>1$, $R_j(\text{gOSN})<1$ for all~$j$, and $R_0^W<1$; exists iff $R_0>1$.
  \item \textbf{RFE} (withdrawal route active, no strains): stable iff $R_j(\text{RFE})<1$ for all~$j$; exists iff $R_0^W>1$.
  \item \textbf{$E_{1g}$, $E_{2g}$} (single-strain endemic, gOSN branch): $E_{jg}$ is
   stable iff $R_j(\text{gOSN})>1$, $R_{3-j}(\text{gOSN})<1$, and $R_0^W<1$; exists iff $R_j(\text{gOSN})>1$.
  \item \textbf{$E_{1}$, $E_{2}$} (single-strain endemic, RFE branch): $E_{jR}$ is stable iff $R_j(\text{RFE})>1$ and $R_{3-j}(\text{RFE})<1$; exists iff $R_j(\text{RFE})>1$.
  \item \textbf{$\text{EE}_g$} (both-strain endemic, gOSN branch): stable iff $R_j(\text{gOSN})>1$ for all~$j$ and $R_0^W<1$; exists iff $R_j(\text{gOSN})>1$ for all~$j$.
  \item \textbf{EE} (full endemic, RFE branch): stable iff $R_j(\text{RFE})>1$ for all~$j$; exists iff $R_j(\text{RFE})>1$ for all~$j$.
\end{itemize}
\fi
The LAS and existence conditions are summarized in Table~\ref{tab:OSN_equilibria_om0},
which gives for each fixed point
the variables that are zero in its coordinate representation (``invaders''),
those that are strictly positive (``residents''), existence conditions,
and extra stability conditions (excluding the existence conditions).
Relay pairs may be read from right to left and downwards.

We denote by $R_j(E), E\in\{gOSN,RFE\}$ the \repF\ of strain~$j$ evaluated at equilibrium~$E$,
and write $R_0 = R(\text{DFE})$ for the basic reproduction number at the DFE.

\begin{table}[H]
  \centering
  \small
  \caption{Equilibria of the OSN model ($\omega=0$): invaders, residents, existence and extra stability conditions.}
  \label{tab:OSN_equilibria_om0}
  \begin{tabular}{|c|c|c|c|c|}
  \hline
  \textbf{Name} & \textbf{Invaders} & \textbf{Residents} & \textbf{Existence} &
  \textbf{Extra Stability Conditions} \\
  \hline\hline
  DFE & $U,W,S_1,B_1,S_2,B_2$ & $x_1$ & always & $R_0<1$ \\
  \hline \hline
  gOSN & $W,S_1,B_1,S_2,B_2$ & $x_1,U$ & $R_0>1$ & $\max_j R_j(\text{gOSN})<1$,\ $R_0<1+\fr{\beta}{\beta_w}$ \\
  \hline
  $E_{1g}$ & $W,S_2,B_2$ & $x_1,U,S_1,B_1$ & $R_1(\text{gOSN})>1$ & $R_2(\text{gOSN})<1$,\ $R_0<1+\fr{\beta}{\beta_w}$ \\
  \hline
  $E_{2g}$ & $W,S_1,B_1$ & $x_1,U,S_2,B_2$ & $R_2(\text{gOSN})>1$ & $R_1(\text{gOSN})<1$,\ $R_0<1+\fr{\beta}{\beta_w}$ \\
  \hline
  $\text{EE}_g$ & $W$ & $x_1,U,S_1,B_1,S_2,B_2$ & $\max_j R_j(\text{gOSN})>1$ & $R_0<1+\fr{\beta}{\beta_w}$ \\
  \hline \hline
  RFE & $S_1,B_1,S_2,B_2$ & $x_1,U,W$ & $R_0>1+\fr{\beta}{\beta_w}$ & $\max_j R_j(\text{RFE})<1$ \\
  \hline
  $E_1$ & $S_2,B_2$ & $x_1,U,W,S_1,B_1$ & $R_1(\text{RFE})>1$ & $R_2(\text{RFE})<1$ \\
  \hline
  $E_2$ & $S_1,B_1$ & $x_1,U,W,S_2,B_2$ & $R_2(\text{RFE})>1$ & $R_1(\text{RFE})<1$ \\
  \hline
  EE & $\emptyset$ & $x_1,U,W,S_1,B_1,S_2,B_2$ & $\max_j R_j(\text{RFE})>1$ &
   none,  when  exists \\
  \hline
  \end{tabular}
\end{table}

\beR
Every threshold in the Extra Stability column also appears in the Existence column
in one or several  relay-successor pairs (
  $E_{1g}$,
  $E_{2g}$  and RFE have two relays each, and gOSN has three). See Figure \ref{fig:relay0}  for the induced relay graph.
The two endemic branches (g-type and R-type) are separated by the threshold
$R_0=1+\fr{\beta}{\beta_w}$: the g-branch is active when $R_0<1+\fr{\beta}{\beta_w}$
and the R-branch when $R_0>1+\fr{\beta}{\beta_w}$.

\eeR

\beD[Full and multiple boundary transcritical relay, siphon-steps ]\label{d:relay-type}
A  boundary transcritical relay  $E^*\to\tilde E$  is called \emph{full} if $E^*$ has exactly one instability
direction, whose threshold coincides with the existence threshold of $\tilde E$
(so $E^*$ is LAS if and only if $\tilde E$ does not exist).

$E^*$ is called \emph{multiple  boundary transcritical relay}
if $E^*$ has \emph{several} possible relay successors/instability directions
 $E^*\to\tilde E_i$, which will be called siphon-steps or relay-successor pairs.

\eeD

Table~\ref{tab:OSN_equilibria_om0} may be visualized via a ``relay graph":
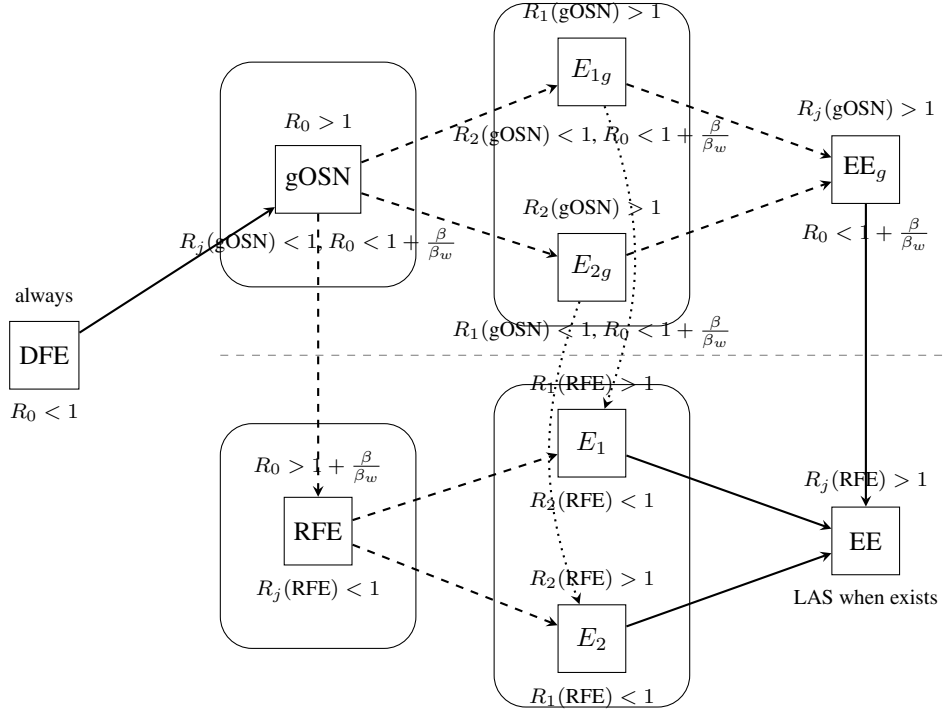
\begin{figure}[H]
\centering
\begin{tikzpicture}[scale=1.3,
    node distance=2cm,
    box/.style={rectangle, draw, minimum size=0.9cm, font=\small},
    lbl/.style={font=\scriptsize, align=center},
    arrow/.style={->, thick, >=stealth},
    darrow/.style={->, thick, >=stealth, dashed},
    dotarrow/.style={->, thick, >=stealth, dotted}
]
% Nodes
\node[box] (DFE)  at (0,    0)    {DFE};
\node[box] (gOSN) at (2.8,  1.8)  {gOSN};
\node[box] (RFE)  at (2.8, -1.8)  {RFE};
\node[box] (E1g)  at (5.6,  2.9)  {$E_{1g}$};
\node[box] (E2g)  at (5.6,  0.9)  {$E_{2g}$};
\node[box] (EEg)  at (8.4,  1.9)  {$\text{EE}_g$};
\node[box] (E1)   at (5.6, -0.9)  {$E_{1}$};
\node[box] (E2)   at (5.6, -2.9)  {$E_{2}$};
\node[box] (EE)   at (8.4, -1.9)  {EE};

% Grouping boxes
\draw[rounded corners=12pt] (1.8, 0.7) rectangle (3.8, 3.0);   % gOSN group
\draw[rounded corners=12pt] (1.8,-3.0) rectangle (3.8,-0.7);   % RFE group
\draw[rounded corners=12pt] (4.6,  0.3) rectangle (6.6, 3.6);  % Ejg group
\draw[rounded corners=12pt] (4.6, -3.6) rectangle (6.6,-0.3);  % EjR group

% Separator between gOSN and RFE (horizontal axis)
\draw[gray, dashed] (1.8,0) -- (9.2,0);

% Existence labels (above nodes)
\node[lbl, above=0.05cm of DFE]  {always};
\node[lbl, above=0.05cm of gOSN] {$R_0>1$};
\node[lbl, above=0.05cm of RFE]  {$R_0>1+\fr{\beta}{\beta_w}$};
\node[lbl, above=0.05cm of E1g]  {$R_1(\text{gOSN})>1$};
\node[lbl, above=0.05cm of E2g]  {$R_2(\text{gOSN})>1$};
\node[lbl, above=0.05cm of EEg]  {$R_j(\text{gOSN})>1$};
\node[lbl, above=0.05cm of E1]  {$R_1(\text{RFE})>1$};
\node[lbl, above=0.05cm of E2]  {$R_2(\text{RFE})>1$};
\node[lbl, above=0.05cm of EE]   {$R_j(\text{RFE})>1$};

% Stability labels (below nodes)
\node[lbl, below=0.05cm of DFE]  {$R_0<1$};
\node[lbl, below=0.05cm of gOSN] {$R_j(\text{gOSN})<1$,\ $R_0<1+\fr{\beta}{\beta_w}$};
\node[lbl, below=0.05cm of RFE]  {$R_j(\text{RFE})<1$};
\node[lbl, below=0.05cm of E1g]  {$R_2(\text{gOSN})<1$,\ $R_0<1+\fr{\beta}{\beta_w}$};
\node[lbl, below=0.05cm of E2g]  {$R_1(\text{gOSN})<1$,\ $R_0<1+\fr{\beta}{\beta_w}$};
\node[lbl, below=0.05cm of EEg]  {$R_0<1+\fr{\beta}{\beta_w}$};
\node[lbl, below=0.05cm of E1]  {$R_2(\text{RFE})<1$};
\node[lbl, below=0.05cm of E2]  {$R_1(\text{RFE})<1$};
\node[lbl, below=0.05cm of EE]   {LAS when exists};

% Arrows (within gOSN branch)
\draw[arrow]  (DFE)  -- (gOSN);
\draw[darrow] (gOSN) -- (E1g);
\draw[darrow] (gOSN) -- (E2g);
\draw[darrow] (E1g)  -- (EEg);
\draw[darrow] (E2g)  -- (EEg);
% Arrows (within RFE branch)
\draw[darrow] (gOSN) -- (RFE);
\draw[darrow] (RFE)  -- (E1);
\draw[darrow] (RFE)  -- (E2);
\draw[arrow]  (E1)   -- (EE);
\draw[arrow]  (E2)   -- (EE);
% Cross-branch arrows (W-invasion from gOSN-branch nodes to RFE-branch nodes)
\draw[dotarrow] (E1g) to[bend left=20]  (E1);
\draw[dotarrow] (E2g) to[bend right=20] (E2);
\draw[arrow]    (EEg) -- (EE);
\end{tikzpicture}
\caption{Relay graph for the nine equilibria ($\omega=0$; see Definition~\ref{d:relay-type}).
Existence conditions appear above each node and extra stability conditions below.
Solid arrows: \emph{full} relays (single threshold simultaneously creates successor and
destabilises predecessor): DFE$\to$gOSN, $E_1\to$EE, $E_2\to$EE, $\mathrm{EE}_g\to$EE.
Dashed arrows: \emph{multiple} relays (predecessor has multiple instability directions):
gOSN$\to E_{jg}$, gOSN$\to$RFE, $E_{jg}\to\mathrm{EE}_g$, RFE$\to E_j$.
Dotted arrows: cross-branch multiple relay s driven by W-invasion ($R_0>1+\fr{\beta}{\beta_w}$):
$E_{1g}\to E_1$, $E_{2g}\to E_2$, $\mathrm{EE}_g\to\mathrm{EE}$.
 %gOSN has three relay directions, and each of $E_{1g}$, $E_{2g}$, RFE has two relay
 %directions.
}
\label{fig:relay0}
\end{figure}

{The mechanism behind the relay mechanisms  in
Table~\ref{tab:OSN_equilibria_om0}}
will be explained in the following section.

\ssec{Explanation of Table~\ref{tab:OSN_equilibria_om0}}

The relay structure of Table~\ref{tab:OSN_equilibria_om0} rests on three ingredients:
\begin{enumerate}[label=(\roman*)]
\item \emph{Forward invariance} of faces implies block factorizations which yield
transversal Jacobian blocks (Definition~\ref{d:transblock}) admitting \regS.

\beD[transversal Jacobian block of a siphon]\label{d:transblock}
Let $\sigma$ be a siphon with face $\mathcal{F}_\sigma:=\{x_\sigma=0\}$.
Since $\mathcal{F}_\sigma$ is forward-invariant, $f_\sigma$ must vanish on
$\mathcal{F}_\sigma$ and may be factored (non-uniquely) as
$f_\sigma(x)=M_\sigma(x)\,x_\sigma$.
A smooth matrix-valued function $x\mapsto M_\sigma(x)$, defined for all
$x\in\mathcal{F}_\sigma$ as above, will be called a
\emph{transversal Jacobian block} of $\sigma$.
\eeD

\beR[Non-uniqueness of $M_\sigma$]
The factorization $f_\sigma(x)=M_\sigma(x)\,x_\sigma$ is not unique off the face
$\mathcal{F}_\sigma$: any smooth matrix $M_\sigma(x)+N(x)$ with
$N(x)\,x_\sigma\equiv 0$ gives an equally valid factorization.
However, on the face $\mathcal{F}_\sigma$, the value
$M_\sigma(x)\big|_{x_\sigma=0}=\partial f_\sigma/\partial x_\sigma\big|_{x_\sigma=0}$
is uniquely determined (it is the $(\sigma,\sigma)$ block of $Df(x)$), so the spectral
abscissa $\alpha(M_\sigma(E^*))$ and its sign, the only quantities needed below,
are independent of any off-face extension.
Thus ``transversal Jacobian block'' is justified in the same sense that
``basic reproduction number'' is: a heuristic name for a quantity whose
\emph{relevant properties} are well-defined even when the underlying
definition is not unique.
\eeR

For example, for $\sigma=\{S_j,B_j\}$ with face $\Pi_j=\{S_j=B_j=0\}$,
the OSN equations
\[
  \dot S_j = f_j(U,B_j) - \gamma_j S_j,\qquad
  \dot B_j = \gamma_j S_j - \mu_j B_j,\qquad j=1,2,
\]
may be written in \emph{block-factored} form
\begin{equation}\label{eq:OSN-block-factor}
  \begin{pmatrix}\dot S_j\\\dot B_j\end{pmatrix}
  \;=\;
  M_j\,
  \begin{pmatrix}S_j\\B_j\end{pmatrix},
  \qquad
  M_j\;:=\;
  \begin{pmatrix}-\gamma_j & \hat f_j(U,B_j)\\\gamma_j & -\mu_j\end{pmatrix},
\end{equation}
where $f_j(U,B_j)=B_j\,\hat f_j(U,B_j)$ and
$\hat f_j(U,0)=(\partial_{B_j}f_j)(U,0)>0$.
On $\Pi_j$ this gives
\begin{equation}\label{eq:Mj}
  M_j(x)
  \;=\;
  \begin{pmatrix}-\gamma_j & (f_j)_B(U,0)\\\gamma_j & -\mu_j\end{pmatrix}=\begin{pmatrix}0 & \hat f_j(U,0)\\0 & 0\end{pmatrix}-\begin{pmatrix}\gamma_j & 0\\-\gamma_j & \mu_j\end{pmatrix},
  \qquad x\in\Pi_j,
\end{equation}
which is a \regS, since the inverse of the second matrix is \nne, and the first matrix is \nne.

\item \emph{Sign equivalence} between transversal spectral abscissas and basic
reproduction numbers, which holds for any ME-type siphon (cf. \cite{Varga,Van}  and our definition \ref{d:VdD}):
\begin{equation}\label{eq:OSN-invasion}
  \operatorname{sign}\bigl(\alpha(M_\sigma(E^*))\bigr)
  \;=\;
  \operatorname{sign}\bigl(R_\sigma(E^*)-1\bigr).
\end{equation}

\item[(iii)] Tangential Hurwitz condition at  resident equilibria.
Let $E^*\in\mathcal F_\Sigma$ be a face equilibrium, and assume it is
face-Hurwitz (see Definition below), which ensures
that all instabilities at $E^*$ are necessarily transversal to
$\mathcal F_\Sigma$.

For a given cover $\Sigma'\prec\Sigma$ with invader block
$\sigma:=\Sigma\setminus\Sigma'$, assume that the corresponding invasion number
satisfies $\al_\sigma(E^*)=0$, while all other invasion numbers remain strictly
subcritical, $\al_{\sigma'}(E^*)<0$ for $\sigma'\neq\sigma$.
Then $Df(E^*)$ has a simple zero eigenvalue associated with the block $\sigma$,
and no other eigenvalues on the imaginary axis.
Consequently, the boundary transcritical bifurcation theorem applies in the
$\sigma$--direction, yielding a locally unique equilibrium branch entering the
adjacent face $\mathcal F_{\Sigma'}$.
If several invasion numbers cross $1$ simultaneously, the center dimension exceeds
one and a higher--codimension analysis is required.

\end{enumerate}

\beD [face-Hurwitzness] We say that $E^*\in\mathcal F_\Sigma$ is \emph{face--Hurwitz} if the Jacobian of the
face-restricted dynamics on $\mathcal F_\Sigma$ (i.e.\ the tangential block of $Df(E^*)$) is Hurwitz.\eeD

Ingredients~(i) and~(ii) are automatic consequences of the ME-type siphon structure
and require no model-specific verification.
The only condition requiring work for each relay step is~(iii): checking that tangential
Hurwitzness is tantamount to the
Extra Stability conditions of Table~\ref{tab:OSN_equilibria_om0}.
The following theorem makes this more precise.

\beT[Relay across a distance-one siphon step]\label{thm:relay-abstract}
Let $\dot x = f(x)$ be a smooth positive ODE on $\mathbb{R}^n_{\ge 0}$ with minimal
siphons $\sigma_1,\ldots,\sigma_k$.
A \emph{distance-one cover} in the siphon lattice is a pair $(\Sigma',\Sigma)$ with
$\Sigma = \Sigma'\cup\sigma$ for some minimal siphon $\sigma\not\subseteq\Sigma'$
and no siphon strictly between $\Sigma'$ and $\Sigma$; we call $\sigma$ the
\emph{invader block}.
For a face equilibrium $E^*\in\mathcal{F}_\Sigma$ we distinguish three classes of variables:
\begin{itemize}
\item \emph{Resident variables}: $\{x_i:i\notin\Sigma\}$, strictly positive at $E^*$.
\item \emph{Latent variables}: $x_{\Sigma'}:=\{x_i:i\in\Sigma'\}$, zeroed at $E^*$
  and still zeroed at the relay successor $\tilde{E}\in\mathcal{F}_{\Sigma'}$.
\item \emph{Non-invader variables}: $\{x_i:i\notin\sigma\}$ = resident $\cup$ latent.
\end{itemize}
Let $M_\sigma(E^*)$ be the transversal Jacobian block
\textup{(}Definition~\ref{d:transblock}\textup{)} and $R_\sigma(E^*)$ the invasion
number \textup{(}Definition~\ref{d:reF}\textup{)}.
Assume:
\begin{enumerate}[label=\textup{(H\arabic*)}]
\item $\sigma$ is a simple ME-type siphon \textup{(}Definitions~\ref{d:VdD}--\ref{d:simpM}\textup{)};
  in particular, $M_\sigma(E^*)$ has a simple zero eigenvalue when $R_\sigma(E^*)=1$.
\item The tangential block $J_{\mathrm{tang}}(E^*)$---the restriction of $Df(E^*)$
  to the non-invader variables $\{x_i:i\notin\sigma\}$---is Hurwitz.
\end{enumerate}
Then:
\begin{enumerate}[label=\textup{(\alph*)}]
\item \emph{(Relay-pair identity.)}
  $E^*$ is transversally unstable in the $\sigma$-direction if and only if
  $R_\sigma(E^*)>1$, which coincides with the existence threshold of the relay
  successor $\tilde{E}$ with $x_\sigma>0$.
\item \emph{(Transcritical bifurcation.)}
  At $R_\sigma(E^*)=1$ there is a transcritical bifurcation transverse to
  $\mathcal{F}_\Sigma$: a branch with $x_\sigma>0$ emerges into $\mathcal{F}_{\Sigma'}$
  and $E^*$ simultaneously loses transversal stability.
\end{enumerate}
\eeT

\begin{proof}
\textit{Block structure.}
Since $\sigma$ is a minimal siphon, $\mathcal{F}_\sigma:=\{x_\sigma=0\}$ is
forward-invariant (Theorem~\ref{s:wl}), so $f_\sigma(x)=M_\sigma(x)\,x_\sigma$.
Differentiating at $E^*$ (where $x_\sigma^*=0$) yields $\partial_y f_\sigma|_{E^*}=0$
for all non-invader $y$, hence $Df(E^*)$ is block lower-triangular:
\[
  Df(E^*)=\begin{pmatrix}M_\sigma(E^*) & 0 \\ * & J_{\mathrm{tang}}(E^*)\end{pmatrix}.
\]
By~(H2), $\alpha(J_{\mathrm{tang}})<0$, so stability is governed solely by
$\alpha(M_\sigma(E^*))$.

\textit{Part~(a).}
By~(H1), $M_\sigma(E^*)$ is Metzler with regular splitting $M_\sigma=F-V$.
The NGM theorem \cite{Van} gives
$\operatorname{sign}(\alpha(M_\sigma(E^*)))=\operatorname{sign}(R_\sigma(E^*)-1)$.

\textit{Part~(b).}
At $R_\sigma(E^*)=1$, the zero eigenvalue of $M_\sigma(E^*)$ is simple by~(H1);
combined with~(H2), $Df(E^*)$ has exactly one simple zero eigenvalue.
Since $\mathcal{F}_\Sigma$ is forward-invariant, the centre manifold is transverse to
$\mathcal{F}_\Sigma$, and the standard boundary transcritical bifurcation theorem
\cite{Van,Boldin} yields a branch with $x_\sigma>0$ entering
$\mathcal{F}_{\Sigma'}$.
\end{proof}

\beR
Three facts hold automatically and need not be listed as hypotheses:
\textup{(i)} face invariance of $\mathcal{F}_\Sigma$ (Theorem~\ref{s:wl});
\textup{(ii)} the block lower-triangular structure of $Df(E^*)$;
\textup{(iii)} the sign equivalence~\eqref{eq:OSN-invasion} for any ME-type siphon.
The only genuine requirements are~(H1) (simple-ME structure) and~(H2) (tangential
Hurwitzness).
\eeR

\beR[Resident vs.\ non-invader variables]
Resident and non-invader variables coincide only for first-level relay steps
($\Sigma'=\emptyset$, $\Sigma$ itself minimal).
For higher-level steps ($\Sigma'\neq\emptyset$), the latent variables $x_{\Sigma'}$
are non-invader but not resident, and (H2) requires Hurwitzness of the full
non-invader block---including the latent sub-block.
In the OSN relay $\mathrm{gOSN}\to E_{1g}$ ($\sigma=\{S_1,B_1\}$, latent
$=\{W,S_2,B_2\}$), (H2) holds under $R_2(\mathrm{gOSN})<1$ (controls $\{S_2,B_2\}$)
and $R_0^W<1$ (controls $\{W\}$).
\eeR

\beR[Disjoint minimal siphons make every step a cover]
When minimal siphons are pairwise disjoint---as in the OSN model---the
``no siphon strictly between'' clause is automatic: any siphon $X$ with
$\Sigma'\subseteq X\subseteq\Sigma'\cup\sigma$ must contain a minimal
$\sigma_l\not\subseteq\Sigma'$, forcing $\sigma_l\cap\sigma\neq\emptyset$, hence
$\sigma_l=\sigma$ and $X=\Sigma'\cup\sigma$.
Thus every pair $(\Sigma',\Sigma'\cup\sigma)$ with $\sigma$ minimal is a cover.
\eeR

Theorem~\ref{thm:relay-abstract} applies to the OSN model as follows.
All OSN siphons are simple ME-type, so (H1) is automatic.
The only content is (H2), which reduces to the Extra Stability conditions in
Table~\ref{tab:OSN_equilibria_om0}.

\beT[Relay structure of the OSN model with omega=0]\label{thm:relay}
Let $E^*$ be any face equilibrium of the OSN system with $\omega=0$ appearing in
Table~\ref{tab:OSN_equilibria_om0}, and let $\sigma=\{S_j,B_j\}$ or $\sigma=\{W\}$
be an absent siphon block at $E^*$.
\begin{enumerate}[label=\textup{(\alph*)}]
\item \emph{(Relay-pair identity.)}
  $E^*$ is transversally unstable in the $\sigma$-direction if and only if
  $R_\sigma(E^*)>1$, which is the existence condition for the relay successor $\tilde E$.
\item \emph{(Transcritical bifurcation.)}
  Under the Extra Stability conditions of $E^*$'s row, at $R_\sigma(E^*)=1$
  a branch with $x_\sigma>0$ bifurcates from $E^*$ into the positive orthant,
  and $E^*$ simultaneously loses LAS.
\item \emph{(LAS of the relay successor.)}
  $\tilde E$ is LAS if and only if the Extra Stability conditions in its row hold.
  In particular, $\mathrm{EE}$ and $\mathrm{EE}_g$ are always LAS when they exist.
\end{enumerate}
\eeT

\begin{proof}
All OSN siphon blocks $\{S_j,B_j\}$ and $\{W\}$ are simple ME-type
(Definitions~\ref{d:VdD}--\ref{d:simpM}), so~(H1) of
Theorem~\ref{thm:relay-abstract} holds automatically for every relay step.
Parts~(a) and~(b) then follow from Theorem~\ref{thm:relay-abstract}(a)--(b)
once~(H2) is verified; part~(c) follows by applying part~(a) to $\tilde E$.
It remains only to verify~(H2)---tangential Hurwitzness---entry by entry;
this is done below.
\end{proof}

\beR
Parts~(a)--(b) are independent of $\omega$: the strain equations~\eqref{eq:OSN-block-factor}
contain no $\omega$, so the block factorization, invasion formula, and sign
equivalence hold for any $\omega\ge0$.
For the withdrawal relay $\mathrm{gOSN}\to\mathrm{RFE}$, on the strain-free face
the $(W,R)$-block at gOSN is upper-triangular with spectral abscissa
$\beta_w U^*-\mu=\mu(R_0^W-1)$, so this relay too is governed by $R_0^W$ for all
$\omega\ge0$.
\eeR

\paragraph{Verification of tangential Hurwitzness~(H2) entry by entry.}

\smallskip
\noindent\emph{(a) Relay $\mathrm{gOSN}\to E_{jg}$, $j=1,2$.}
The gOSN equilibrium is
\[
  E^*_{\mathrm{gOSN}}
  \;=\;
  \Bigl(x_1^*,\,U^*,\,0,\,0,\,0,\,0,\,0\Bigr),
  \qquad
  x_1^* = \frac{\mu_n}{\beta},\quad U^* = \frac{\mu(R_0-1)}{\beta},
\]
which exists iff $R_0>1$.
The invader block is $\sigma=\{S_j,B_j\}$; the non-invader variables are
$\{x_1,U,W,S_{3-j},B_{3-j}\}$.
The tangential block $J_{\mathrm{tang}}$ is block upper-triangular
(resident variables $y$ placed first, invader block second;
opposite to the convention of Theorem~\ref{thm:relay-abstract}):
\[
  J^* = \begin{pmatrix} J_{\mathrm{tang}} & * \\ 0 & M_j(E^*) \end{pmatrix}.
\]
Under the Extra Stability conditions $R_{3-j}(\mathrm{gOSN})<1$ and
$R_0<1+\beta/\beta_w$, the block $J_{\mathrm{tang}}$ is Hurwitz: the condition
$R_{3-j}<1$ controls the $\{S_{3-j},B_{3-j}\}$ sub-block and $R_0<1+\beta/\beta_w$
controls $\{W\}$.
This verifies (H2), completing the proof of parts~(a)--(b) for this relay.
The right null-vector of $M_j(E^*)$ at $R_j=1$ is
$v_j=(\mu_j,\gamma_j)^\top\gg0$, confirming the branch enters
$\{S_j>0,B_j>0\}$.

\smallskip
\noindent\emph{(b) Relay $\mathrm{gOSN}\to\mathrm{EE}_g$.}
When $\max_j R_j(\mathrm{gOSN})>1$, both strain blocks are unstable at gOSN;
sequential relays via $E_{1g}$ or $E_{2g}$ reach $\mathrm{EE}_g$.
The Extra Stability condition $R_0<1+\beta/\beta_w$ for $\mathrm{EE}_g$ ensures
$W$ remains uninvadable (H2 for the $\{W\}$ block at $\mathrm{EE}_g$).

\smallskip
\noindent\emph{(c) Relays $\mathrm{RFE}\to E_j\to\mathrm{EE}$.}
At RFE, $U^*=\mu/\beta_w$.  The same argument as in~(a) applies with gOSN replaced
by RFE and $R_j(\mathrm{RFE})$ in place of $R_j(\mathrm{gOSN})$.
The terminal equilibrium EE carries no Extra Stability condition: (H2) is vacuous
there, and LAS follows from part~(a) applied to EE as the relay successor.

\smallskip
\noindent\emph{(a) Relay $\mathrm{gOSN}\to E_{jg}$, $j=1,2$ (boundary--compatible transcritical).}
The only point that requires care is that $\mathrm{gOSN}$ lies on the invariant face
$\Pi_j=\{S_j=B_j=0\}$, so the bifurcating branch must be shown to cross \emph{from the
face into the positive orthant}. This is a standard ``transcritical bifurcation on an
invariant coordinate face'' and may be proved directly from the center--manifold
reduction, without appealing to an ``interior'' theorem.

\medskip
\noindent\textbf{Step 3(a.1): reduce to a one--dimensional center variable.}
Fix $j\in\{1,2\}$ and consider parameters near a threshold where
$R_j(\mathrm{gOSN})=1$. Under the ``Extra Stability'' hypotheses in
Table~\ref{tab:OSN_equilibria_om0} (namely $R_{3-j}(\mathrm{gOSN})<1$ and
$R_0<1+\beta/\beta_w$), all eigenvalues of the Jacobian at $\mathrm{gOSN}$ except those
in the $(S_j,B_j)$ block have negative real part. By~(H1) of Theorem~\ref{thm:relay-abstract}, at $R_j(\mathrm{gOSN})=1$
the transversal block $M_j(E^*_{\mathrm{gOSN}})$ has a \emph{simple} eigenvalue $0$ and
one stable eigenvalue $-(\gamma_j+\mu_j)$; hence the full Jacobian has a one--dimensional
center subspace and the center manifold is one--dimensional and $C^2$ (e.g.\
\cite[Ch.~3]{Carr1981} or \cite[Ch.~1]{Wiggins2003}).

Let $q_j\gg0$ be the (right) eigenvector of $M_j(E^*_{\mathrm{gOSN}})$ associated with
the eigenvalue $0$ (one may choose $q_j$ with positive components since $M_j$ is Metzler
at the face). Introduce a scalar center coordinate $x$ along the positive ray
$(S_j,B_j)=x\,q_j$.

\medskip
\noindent\textbf{Step 3(a.2): invariance forces the ``$x$ factor'' (this is the boundary part).}
Because the face $\Pi_j$ is forward invariant, \eqref{eq:OSN-block-factor} implies that the
vector field in the $(S_j,B_j)$--variables vanishes at $(S_j,B_j)=(0,0)$, and therefore
the reduced center dynamics must have the form
\begin{equation}\label{eq:cm-factor}
\dot x = x\,\Phi(x,\eta),
\qquad \eta:=R_j(\mathrm{gOSN})-1,
\end{equation}
with $\Phi$ smooth. In particular, $x\equiv0$ is an equilibrium for all $\eta$, which
encodes the fact that the equilibrium branch on the face persists.

\medskip
\noindent\textbf{Step 3(a.3): normal form and the crossing into the positive orthant.}
Taylor expanding $\Phi$ at $(x,\eta)=(0,0)$ yields the normal form
\begin{equation}\label{eq:cm-normal-relay}
\dot x = x\,(a\,\eta + b\,x + o(|\eta|+|x|)),
\end{equation}
where $a\neq0$ because $\partial_\eta \alpha(M_j(E^*_{\mathrm{gOSN}}))\neq0$ at the
threshold (equivalently, the transversal eigenvalue crosses $0$ with nonzero speed),
and $b\neq0$ is the generic nondegeneracy coefficient (it may be computed from second
derivatives of the original vector field; see \cite[\S3]{Carr1981} or
\cite[\S3.4]{Kuznetsov2013}). Equation \eqref{eq:cm-normal-relay} has precisely two
equilibria near $(0,0)$:
\[
x=0
\qquad\text{and}\qquad
x_+(\eta)= -\frac{a}{b}\,\eta + o(\eta).
\]
Since $q_j\gg0$, the condition $x_+(\eta)>0$ is equivalent to $(S_j,B_j)>0$, i.e.\ the
branch enters the \emph{positive orthant} rather than leaving it. Thus, on the side of
the threshold where $x_+(\eta)>0$, one obtains a unique equilibrium branch with
$S_j,B_j>0$ bifurcating from $\mathrm{gOSN}$; this is exactly the equilibrium $E_{jg}$
listed in Table~\ref{tab:OSN_equilibria_om0}. Moreover, the linearization of
\eqref{eq:cm-normal-relay} shows the exchange of stability: as $\eta$ changes sign,
the equilibrium on the face ($x=0$, i.e.\ $\mathrm{gOSN}$) loses transversal stability
and the interior branch ($x=x_+(\eta)$, i.e.\ $E_{jg}$) becomes stable along the center
direction.

\medskip
\noindent\textbf{Intuition.}
The boundary aspect is entirely encoded in the factorization \eqref{eq:cm-factor}:
invariance forces the reduced vector field to vanish at $x=0$ for all parameters, so
when the transversal growth rate changes sign (controlled by $\eta=R_j(\mathrm{gOSN})-1$),
a second root of $\dot x$ must appear generically. The sign of this root determines
whether the new equilibrium lies in the admissible region ($x>0$). No conservation law
is used; the mechanism is local but it applies at \emph{every} invariant face
equilibrium, not only at the DFE.

\medskip
\noindent\textbf{References.}
A full proof of the reduction and normal form is standard; see the center--manifold and
bifurcation treatments in \cite{Carr1981,Wiggins2003,Kuznetsov2013}. In population and
epidemiological models, the same ``invasion through a boundary equilibrium'' mechanism
(including the critical case and direction of the bifurcation) is discussed explicitly
in \cite{Boldin}.

\section{Algorithmic detection of relay transitions along the siphon lattice}\lbl{s:alg}

\subsection{Relay test along a distance--one step in a siphon lattice: algorithm and  code}

\paragraph{Purpose.}
Given a positive ODE $\dot x=f(x,p)$ on $\mathbb R^n_{\ge0}$ (with rational parameter point $p$),
and a distance-one cover $(\Sigma',\Sigma)$ in the siphon lattice (invader block $\sigma=\Sigma\setminus\Sigma'$,
see Theorem~\ref{thm:relay-abstract}), the goal is to
decide whether a \emph{relay} occurs from a face equilibrium $E^*\in\mathcal{F}_\Sigma$ to a face
equilibrium on the adjacent (less constrained) face $\mathcal{F}_{\Sigma'}$.

\subsection{Distance--one does not mean a one--species invader}
In the OSN table (for $\omega=0$), the relay $\mathrm{gOSN}\to E_{1g}$ frees the \emph{block}
$(S_1,B_1)$ at once, not a single variable. Thus the ``distance--one'' move must be interpreted
at the level of the \emph{siphon lattice}, not the number of freed coordinates.

We use the lattice terminology of Theorem~\ref{thm:relay-abstract}: a
\emph{distance-one cover} $(\Sigma',\Sigma)$ in the siphon lattice has
\emph{invader block} $\sigma=\Sigma\setminus\Sigma'$, corresponding face
$\mathcal{F}_\Sigma:=\{x_s=0:s\in\Sigma\}$, and face equilibrium $E^*\in\mathcal{F}_\Sigma$.
In the OSN model the cover $(\{W,S_2,B_2\},\{W,S_1,B_1,S_2,B_2\})$ has invader block
$\sigma=\{S_1,B_1\}$.

\subsection{The shared inequality explaining the ``coincidence''}
Let $E^*\in\mathcal{F}_\Sigma$ be a face equilibrium. Write $x=(x_\sigma,\bar x)$ where
$x_\sigma$ are the invader (freed) coordinates and $\bar x$ are the non-invader coordinates.
Face invariance of $\mathcal{F}_\sigma=\{x_\sigma=0\}$ (because $\sigma$ is a minimal siphon) implies
\[
f_\sigma(\bar x,0,p)\equiv 0.
\]
Hence the transversal linearization at $E^*=(0,\bar x^\ast)$ is the \emph{single matrix}
\begin{equation}\label{eq:transblock}
M_\sigma(E^*):=D_{x_\sigma} f_\sigma(\bar x^\ast,0,p)\in\mathbb R^{|\sigma|\times|\sigma|}.
\end{equation}
In positive/CRN settings, $M_\sigma(E^*)$ is typically Metzler. The \emph{shared inequality} is then
\begin{equation}\label{eq:shared}
\boxed{\ \alpha\!\bigl(M_\sigma(E^*)\bigr)>0\ }\qquad
(\alpha=\text{spectral abscissa}).
\end{equation}
It is ``shared'' because:
\begin{itemize}\setlength{\itemsep}{0pt}
\item it is \emph{exactly} the criterion that $E^*$ is transversally unstable in the freed block $\sigma$
  (Theorem~\ref{thm:relay-abstract}(a));
\item the same loss of invertibility of \eqref{eq:transblock} at $\alpha=0$ is the local mechanism by which
a nearby equilibrium branch on $\mathcal{F}_{\Sigma'}$ may enter with $x_\sigma>0$
(boundary transcritical on the invariant face $\mathcal{F}_\sigma$).
\end{itemize}
Thus the ``coincidence'' is not between two unrelated computations: it is the single condition
\eqref{eq:shared} evaluated at $E^*$.

\beR[OSN ($\omega=0$) dictionary]
In Table~\ref{tab:OSN_equilibria_om0}, for the relay $\mathrm{gOSN}\to E_{1g}$ one has
$\Sigma=\{W,S_1,B_1,S_2,B_2\}$, $\Sigma'=\{W,S_2,B_2\}$, so $\sigma=\Sigma\setminus\Sigma'=\{S_1,B_1\}$.
Then $M_\sigma(\mathrm{gOSN})$ is precisely the $2\times2$ block $M_1(\mathrm{gOSN})$,
and \eqref{eq:shared} is equivalent to $R_1(\mathrm{gOSN})>1$.
\eeR

\subsection{Algorithm}
\paragraph{Input.}
A rational vector field $f(x,p)$ defining $\dot x=f(x,p)$ on $\mathbb R^n_{\ge0}$,
a rational parameter point $p=p_0$, and a distance-one cover $(\Sigma',\Sigma)$ in the siphon lattice
with invader block $\sigma=\Sigma\setminus\Sigma'$.

\paragraph{Output.}
\texttt{RelayHolds}, \texttt{NoRelay}, or \texttt{Undecided} (timeout / non-rational intermediates).

\paragraph{Procedure.}
\begin{enumerate}\setlength{\itemsep}{0pt}
\item[\textbf{(A)}] \textbf{Compute face equilibria exactly.}
Solve $f(x,p_0)=0$ under the constraints $x_i=0$ for $i\in\Sigma$ (resp.\ $i\in\Sigma'$),
retaining only nonnegative solutions. If any intermediate quantity is non-rational, return \texttt{Undecided}.
\item[\textbf{(B)}] \textbf{Choose candidates.}
Pick a face equilibrium $E^*\in\mathcal{F}_\Sigma$. For $\mathcal{F}_{\Sigma'}$, keep only equilibria
$\tilde E\in\mathcal{F}_{\Sigma'}$ with $x_\sigma(\tilde E)\gg 0$ (since $\sigma$ is the freed block).
\item[\textbf{(C)}] \textbf{Compute the shared inequality at $E^*$.}
Form the transversal block $M_\sigma(E^*)$ in \eqref{eq:transblock}. If $M_\sigma(E^*)$ is Metzler,
compute $\alpha(M_\sigma(E^*))$ (exactly when rational; otherwise return \texttt{Undecided}).
If $\alpha(M_\sigma(E^*))\le 0$, return \texttt{NoRelay}.
\item[\textbf{(D)}] \textbf{Confirm by stability of $\tilde E$ (optional but decisive).}
Compute $J(\tilde E)=D_x f(\tilde E,p_0)$ and test Hurwitzness exactly (e.g.\ exact Routh--Hurwitz if coefficients are rational).
If some $\tilde E$ is Hurwitz, return \texttt{RelayHolds}; otherwise return \texttt{NoRelay}.
\end{enumerate}

\paragraph{What explains the coincidence.}
Step \textbf{(C)} is the explanation: the same matrix $M_\sigma(E^*)$ is the linearization governing
transversal stability of $E^*$ and the local bifurcation mechanism by which an equilibrium with $x_\sigma>0$
can branch from the invariant face $\mathcal{F}_\sigma$.

\beR[Termination and timeout]
The algorithm terminates with \texttt{RelayHolds} or \texttt{NoRelay} whenever every
characteristic polynomial coefficient encountered in steps~(A)--(D) is rational in the
parameters.  A \texttt{timeout} or \texttt{Undecided} return reflects symbolic
non-decidability of the relay condition under the given kinetics---it is a principled
limitation of exact arithmetic, not a numerical instability.
\eeR

\paragraph{Relation to existing literature.}
Classical next–generation matrix theory characterizes invasion at the disease–free
equilibrium via a one–dimensional threshold.
In the present work, invasion occurs along minimal siphons, which may involve
several species simultaneously.
The relevant transversal object is therefore a Metzler block
$M_\sigma$ associated with a minimal siphon $\sigma$, and invasion is governed by the
sign of its spectral abscissa $\alpha(M_\sigma)$.
While related block–invasion criteria appear implicitly in persistence theory
and in multi–strain epidemic models, the explicit organization of equilibria
and relays along the lattice generated by minimal siphons appears to be new.
Analytic tools such as resolvent bounds and rank–one perturbation formulas
(e.g.\ Badr--Alexiades) are used here only to control stability once a relay
candidate is identified, not to explain the relay mechanism itself.

\medskip\noindent\textbf{Relation to Hofbauer invasion graphs.}
Invasion graphs in the sense of Hofbauer encode which resident communities
(boundary equilibria) can be invaded by which missing species, based on the sign of
invasion rates evaluated at the resident equilibrium.
This is conceptually close to the relay picture, since both are organized by invasion
numbers evaluated at boundary equilibria.
The relay graphs considered here differ in two structural respects:
\begin{enumerate}\setlength{\itemsep}{2pt}
\item \emph{Nodes ordered by inhabited siphon faces.}
Equilibria are indexed by faces in the siphon lattice, introducing a canonical
partial order and a natural notion of adjacency (distance-one covers $(\Sigma',\Sigma)$),
rather than treating equilibria as an unstructured set of communities.
\item \emph{Edges are canonical cover steps labeled by reproduction functions.}
Each edge corresponds to releasing a minimal siphon block $\sigma$ along a cover
$(\Sigma',\Sigma)$, and its label is the evaluation $R_\sigma(E^*)$.
The same inequality governs simultaneously the transversal instability of the
resident node and the existence condition of the successor node.
\end{enumerate}
From this perspective, Hofbauer's invasion graphs may be viewed as coarse-grained
representations of invasion relations among boundary equilibria, whereas the
siphon-lattice relay graph refines this by imposing the canonical adjacency structure
from invariant faces.
Relay chains correspond to monotone paths in the siphon lattice, and the reproduction
function provides a computable labeling in the ME/CRN class considered here.

\subsection{Exact Mathematica code (succeeds if rational, returns \texttt{Undecided} else)}

\begin{verbatim}
(* Relay test along a distance-one lattice step T < S with freed block U = S\T.
   Works in exact arithmetic when all intermediate expressions are rational.
   Returns "RelayHolds", "NoRelay", or "Undecided". *)

ClearAll[RationalExprQ, ExactHurwitzQ, RelayTestCover];

RationalExprQ[expr_] :=
  FreeQ[Together[expr], _Root | _AlgebraicNumber | _Surd];

(* Exact Hurwitz test for rational matrices via characteristic polynomial + Routh-Hurwitz
   determinants (sufficient for algorithmic use; returns $Failed if non-rational). *)
ExactHurwitzQ[A_] := Module[{lam, poly, coeffs, deg, a, Hdet},
  poly = CharacteristicPolynomial[A, lam] // Expand;
  coeffs = CoefficientList[poly, lam];
  If[!RationalExprQ[coeffs], Return[$Failed]];
  deg = Length[coeffs] - 1;
  If[deg == 0, Return[True]];
  coeffs = Reverse[coeffs]; (* leading to constant *)
  If[coeffs[[1]] == 0, Return[False]];
  If[coeffs[[1]] < 0, coeffs = -coeffs]; (* normalize leading >0 *)
  a[i_] := coeffs[[i + 1]]; (* i=0..deg, a[0]=leading *)

  Hdet[k_] := Det@Table[
     a[2 i - j] /. a[t_] /; t < 0 || t > deg -> 0,
     {i, 1, k}, {j, 1, k}
  ];
  And @@ Table[ Simplify[Hdet[k] > 0], {k, 1, deg} ]
];

(* Main routine.
   RHS: list of n expressions for xdot
   vars: list {x1,...,xn}
   parRules: parameter point substitution
   S,T: lists of indices (1..n) representing siphons with T < S in the generated lattice
*)
RelayTestCover[RHS_List, vars_List, parRules_List, S_List, T_List] := Module[
  {n = Length[vars], U, subsS, subsT, IS, IT, fS, fT, solS, solT,
   ES, ET, J, MU, alphaMU, JT, hur},

  If[!SubsetQ[T, S], Return["BadInput"]];
  U = Complement[S, T];
  If[U === {}, Return["BadInput"]];

  (* Face restrictions *)
  subsS = Thread[vars[[S]] -> 0];
  subsT = Thread[vars[[T]] -> 0];
  IS = Complement[Range[n], S];
  IT = Complement[Range[n], T];
  fS = (RHS /. subsS)[[IS]];
  fT = (RHS /. subsT)[[IT]];

  (* Solve face equilibria exactly *)
  solS = Solve[Thread[(fS /. parRules) == 0], vars[[IS]], Reals];
  solT = Solve[Thread[(fT /. parRules) == 0], vars[[IT]], Reals];

  solS = Select[solS,
    (And @@ Thread[(vars[[IS]] /. #) >= 0]) && RationalExprQ[vars[[IS]] /. #] &
  ];
  solT = Select[solT,
    (And @@ Thread[(vars[[IT]] /. #) >= 0]) && RationalExprQ[vars[[IT]] /. #] &
  ];

  If[solS === {} || solT === {}, Return["NoRelay"]];

  (* Jacobian *)
  J = D[RHS, {vars}];

  (* Try each ES; declare success if some ET confirms *)
  Do[
    ES = Join[subsS, solS[[i]], parRules];

    (* Shared inequality: alpha(M_U(ES)) > 0 *)
    MU = (J[[U, U]] /. ES) // Simplify;
    If[!RationalExprQ[MU], Return["Undecided"]];
    alphaMU = Max[Re[Eigenvalues[MU]]];
    If[!RationalExprQ[alphaMU], Return["Undecided"]];
    If[Simplify[alphaMU <= 0], Continue[]];

    (* Confirm by existence/stability on F_T with x_U>0 *)
    Do[
      ET = Join[subsT, solT[[j]], parRules];
      If[Or @@ Thread[(vars[[U]] /. ET) <= 0], Continue[]];

      JT = (J /. ET) // Simplify;
      If[!RationalExprQ[JT], Return["Undecided"]];
      hur = ExactHurwitzQ[JT];
      If[hur === $Failed, Return["Undecided"]];
      If[hur === True, Return["RelayHolds"]],
      {j, Length[solT]}
    ],
    {i, Length[solS]}
  ];

  "NoRelay"
];
\end{verbatim}

\paragraph{How to read the code.}
The relay ``coincidence'' is implemented by the single test
\[
\alpha\!\bigl(M_\sigma(E^*)\bigr)>0,
\]
computed as the spectral abscissa of the transversal block $J(E^*)_{\sigma,\sigma}$.
The remaining step checks whether the relay-successor $\tilde E$ with $x_\sigma(\tilde E)\gg 0$
exists (via exact solving on $\mathcal{F}_{\Sigma'}$) and is Hurwitz
(via exact Routh--Hurwitz on the full Jacobian).

\beR[The common object]
The ``coincidence'' between existence and transversal stability of a relay pair has a single
explanation: both are governed by the \emph{invader block map}
$x_\sigma\mapsto f_\sigma(\bar x^\ast,x_\sigma)$ on the siphon face $\mathcal{F}_\sigma$.
Transversal stability is decided by the linearization $M_\sigma(E^*)$ at $x_\sigma=0$, and
the existence of the relay successor on $\mathcal{F}_{\Sigma'}$ is governed by the same matrix
losing invertibility (boundary transcritical bifurcation).
The abstract statement is Theorem~\ref{thm:relay-abstract}.
\eeR

\beR[Four structural differences from the classical transcritical bifurcation]
Although each relay transition is locally a transcritical bifurcation, it differs
from the classical transcritical bifurcation in four structural respects, which
explain why relay tables are rigid and algorithmic.
\begin{enumerate}\setlength{\itemsep}{2pt}
\item \textbf{Geometric origin of the critical eigenvalue.}
In the classical setting the zero eigenvalue arises from a generic parameter
unfolding at an ambient-space equilibrium.
In a transcritical relay the critical eigenvalue is \emph{enforced} by the
forward-invariant face $\mathcal{F}_\Sigma$: face invariance forces a block form
of $Df(E^*)$ and the zero eigenvalue belongs to the transversal block $M_\sigma(E^*)$.
The threshold is anchored in face geometry, not in an arbitrary unfolding.

\item \textbf{Block invasion rather than a single direction.}
Classical transcritical bifurcation is commonly presented with a one-dimensional
critical direction.
In a relay the invader is a minimal siphon block $\sigma$, which may contain several
coordinates.
Transversal instability and successor existence are both governed by the Metzler
block $M_\sigma(E^*)$, not by a single coordinate direction.

\item \textbf{Lattice-prescribed successor.}
In classical transcritical bifurcation, two equilibrium branches intersect, but the
bifurcation does not by itself prescribe a global organization.
In a relay the successor equilibrium is constrained to lie on the adjacent face
$\mathcal{F}_{\Sigma'}$ determined by the cover $(\Sigma',\Sigma)$ in the siphon lattice.
This combinatorial constraint is what makes relay graphs canonically defined.

\item \textbf{Shared inequality for instability and existence.}
In the classical normal form, the link between existence of a nontrivial branch and
instability of the trivial branch is a local property of a specific reduced equation.
In relay systems the single invasion inequality $R_\sigma(E^*)>1$ simultaneously
governs (i)~loss of transversal stability of the resident equilibrium on
$\mathcal{F}_\Sigma$ and (ii)~existence of the successor equilibrium on
$\mathcal{F}_{\Sigma'}$.
This coincidence follows from the siphon-induced factorization
$f_\sigma(\bar x,x_\sigma)=M_\sigma(\bar x,x_\sigma)\,x_\sigma$
and the Metzler/regular-splitting structure.
\end{enumerate}
\eeR

\paragraph{Normal form for a siphon-induced transcritical relay.}
\noindent\textbf{Setting.}
Let $\dot z=f(z,\theta)$ be a $C^k$ ($k\ge 3$) vector field on $\mathbb{R}^n_{\ge 0}$,
with bifurcation parameter $\theta\in\mathbb{R}$.
Assume a siphon induces the forward-invariant face
$\mathcal{F}_\Sigma:=\{x_\sigma=0\}$,
with coordinates $z=(\bar x, x_\sigma)$, where $\bar x\in\mathbb{R}^{n-m}$
are the non-invader variables and $x_\sigma\in\mathbb{R}^m$ is the invader block.
Face invariance reads
\begin{equation}\label{eq:inv-face}
f_\sigma(\bar x,0,\theta)\equiv 0,
\qquad\text{hence}\qquad
D_{\bar x} f_\sigma(\bar x,0,\theta)\equiv 0.
\end{equation}
Therefore the Jacobian at any face equilibrium $E^*=(\bar x^*,0)$ is block lower-triangular:
\[
Df(E^*,\theta)=
\begin{pmatrix}
J_{\mathrm{tang}}(\bar x^*,\theta) & 0\\
\ast & M_\sigma(\bar x^*,\theta)
\end{pmatrix},
\qquad
M_\sigma(\bar x^*,\theta):=D_{x_\sigma} f_\sigma(\bar x^*,0,\theta).
\]
Assume $(\bar x^*(\theta),0)$ is a smooth branch of face equilibria.

\smallskip\noindent\textbf{Reduction to a scalar center equation.}
Assume at $\theta=0$:
\begin{enumerate}\setlength{\itemsep}{1pt}
\item $J_{\mathrm{tang}}(\bar x^*(0),0)$ is Hurwitz;
\item $M_\sigma(\bar x^*(0),0)$ has a simple eigenvalue $0$ and all other eigenvalues
  with strictly negative real parts;
\item the remaining spectrum of $Df(E^*,0)$ is strictly stable.
\end{enumerate}
Then there exist local $C^{k-1}$ coordinates $(\eta,\xi)\in\mathbb{R}^{n-1}\times\mathbb{R}$,
with $\xi$ the center coordinate, such that the dynamics is $C^{k-1}$-conjugate to
\begin{equation}\label{eq:center-manifold-form}
\dot\eta = H(\eta,\xi,\theta),\qquad
\dot\xi = a(\theta)\,\xi + b\,\xi^2 + O(|\xi|^3+|\theta|\,|\xi|^2),
\end{equation}
where $a(0)=0$ and $b\neq 0$ (transcritical nondegeneracy).

\smallskip\noindent\textbf{Where the relay structure enters.}
\begin{itemize}\setlength{\itemsep}{2pt}
\item \emph{Transversal threshold.}
The coefficient $a(\theta)$ is the dominant eigenvalue of $M_\sigma(\bar x^*(\theta),\theta)$.
For ME-type siphons, $M_\sigma$ is Metzler and (by sign equivalence):
\[
\operatorname{sign}(a(\theta))
\;=\;
\operatorname{sign}\bigl(R_\sigma(E^*(\theta),\theta)-1\bigr).
\]
The relay threshold is thus a reproduction-function evaluation at a resident
equilibrium, not an arbitrary unfolding parameter.

\item \emph{Factorization and shared inequality.}
Face invariance~\eqref{eq:inv-face} implies $f_\sigma(\bar x,x_\sigma,\theta)
=M_\sigma(\bar x,x_\sigma,\theta)\,x_\sigma$.
The reduced center equation inherits this factorization: the same scalar
inequality governs both (i)~loss of stability of the resident branch $\xi=0$
and (ii)~appearance of the nontrivial branch $\xi(\theta)\neq 0$.
This is the local analytic origin of the relay ``shared inequality''.

\item \emph{Lattice constraint.}
The nontrivial equilibrium branch corresponds to releasing $\sigma$ along the cover
$(\Sigma',\Sigma)$; the successor equilibrium therefore lies on $\mathcal{F}_{\Sigma'}$,
as prescribed by the siphon lattice.
\end{itemize}

%%%%%%%%%%%%%%%%%%%%%%%%%%%%%%%%%%%%%%%%%

\section{Application of the relay algorithm to the case $\omega>0$}\lbl{s:omp}

\beR[Scope of this section]
This section is exploratory: equilibria for $\omega>0$ are no longer available in
closed form, so a complete classification is not attempted.
The key point is that relay \emph{predictions}---which equilibria exist and on which
siphon face, and which invasion inequalities govern their appearance---follow directly
from invasion numbers and the siphon lattice, independently of whether equilibrium
coordinates are explicit.
Explicit coordinates affect only the \emph{verification} of relay predictions, not the
predictions themselves.
The analysis is local and algebraic throughout: it does not address global dynamics
such as basins of attraction or the competition between the gOSN branch and the
RFE branch in the parameter region $R_0>1+\beta/\beta_w$ where both could in principle
coexist (as separate locally attracting sets).
\eeR

When $\omega>0$ the skeptic compartment $R$ is no longer absorbing: it feeds back into
the withdrawn-user compartment via $\dot W = \beta_w U W + \omega R - \mu W$.
Any believer flow $\mu_j B_j$ first creates skeptics ($R>0$), which then cascade into
withdrawn users ($W>0$). At any strain-endemic equilibrium,
\begin{equation}\label{eq:RW-link}
  R^* = \frac{\mu_1 B_1^* + \mu_2 B_2^*}{\omega} > 0,
  \qquad
  W^* = \frac{\omega R^*}{\mu - \beta_w U^*} > 0
  \quad \bigl(\text{when }\beta_w U^* < \mu\bigr).
\end{equation}
Hence, the gOSN-branch equilibria $E_{1g}$, $E_{2g}$, $\mathrm{EE}_g$ that appeared
in Table~\ref{tab:OSN_equilibria_om0} (with $W=0$, $R=0$) are \emph{absent}
for $\omega>0$: the $R\to W$ coupling immediately makes $W>0$ whenever strains are present.
The model therefore has six fixed points instead of nine.

We retain the notation of Section~\ref{s:relO}: $R_j(E)$ denotes the invasion number of
rumor~$j$ at equilibrium~$E$ and
$R_0^W = \frac{\beta_w}{\beta}(R_0-1)$ is the withdrawal invasion number.
The LAS and existence conditions known to date are collected in
Table~\ref{tab:OSN_equilibria_om}; the stability conditions for the strain-endemic
equilibria (marked~$\dagger$) are predicted by Theorem~\ref{thm:relay-abstract}
and remain to be verified by direct Routh–Hurwitz computation.

\ssec{Summary: fixed-point coordinates}\lbl{s:coordsomP}

\noindent\textbf{Part I: Rational equilibria} (all coordinates exact):

\noindent\scalebox{0.87}{$\displaystyle
\begin{array}{@{}l@{\quad}l}
\mathrm{OSND}: &
  (x_1,U,W,R,S_1,B_1,S_2,B_2)
  =\Bigl(\dfrac{\Lambda}{\mu},\;0,\;0,\;0,\;0,\;0,\;0,\;0\Bigr)\\[10pt]
\mathrm{gOSN}: &
  \hat x_1=\dfrac{\mu_n}{\beta},\quad
  \hat U=\dfrac{\beta\Lambda-\mu\mu_n}{\beta\mu_n},\quad
  W=R=S_1=B_1=S_2=B_2=0\\[10pt]
\mathrm{RFE}: &
  \tilde x_1=\dfrac{\Lambda\beta_w}{\mu(\beta+\beta_w)},\quad
  \tilde U=\dfrac{\mu}{\beta_w},\quad
  \tilde W=\dfrac{\beta(\tilde x_1-\hat x_1)}{\beta_w},\quad
  R=S_1=B_1=S_2=B_2=0
\end{array}$}

\bigskip
\noindent\textbf{Part II: Reducible equilibria} ($E_1$, $E_2$, EE: irrational in general).
Common rational substitutions (from \eqref{eq:x1}): once $x_1$ is known,
\begin{equation}\label{eq:ratsubstMP}
W=\frac{\beta(x_1-\hat x_1)}{\beta_w},\quad
B_j=\frac{\gamma_j S_j}{\mu_j},\quad
R=\frac{\gamma_1 S_1+\gamma_2 S_2}{\omega},\quad
U=\frac{\mu}{\beta_w}+\frac{\gamma_1 S_1+\gamma_2 S_2}{\mu_n-\beta x_1}.
\end{equation}
Setting $y:=x_1-\hat x_1>0$, $x_1$ is the positive root of a quadratic (given below).
Existence of a positive root with all coordinates positive is predicted by the relay
structure but \emph{not yet fully verified} for EE.

\medskip
\noindent\underline{$E_1$} ($S_2=B_2=0$, positive root exists iff $R_1(\mathrm{gOSN})>1$):
\[
S_1=\frac{\mu\!\left(\dfrac{\beta x_1}{\mu_n}-1\right)\bigl(\beta_1\mu-\mu_1(\alpha_1\mu+\beta_w)\bigr)}
        {\gamma_1\beta_w(-\alpha_1\mu_1+\beta_1+(\beta x_1-\mu)\epsilon_1)},
\quad S_2=0;\quad
A_1 y^2+B_1 y+C_1=0,
\]
\begin{align*}
A_1&=\beta^2(\beta+\beta_w)\mu\epsilon_1,\\
B_1&=\beta\bigl[\beta\beta_w(\mu_1-\Lambda\epsilon_1)
               +\beta_w\mu(-\alpha_1\mu_1+\beta_1+\epsilon_1\mu_n)
               +\beta\mu\epsilon_1\mu_n\bigr],\\
C_1&=-\beta_w\bigl(-\alpha_1\beta\Lambda\mu_1+\beta\beta_1\Lambda
                   +\alpha_1\mu\mu_1\mu_n-\beta_1\mu\mu_n-\beta\mu_1\mu_n\bigr).
\end{align*}

\medskip
\noindent\underline{$E_2$} ($S_1=B_1=0$): symmetric; replace $1\leftrightarrow 2$ in $A_1,B_1,C_1$ and $S_1$.

\medskip
\noindent\underline{EE} (both strains active). Set
$D_0=(\beta_2-\alpha_2\mu_2)\epsilon_1+(\beta_1-\alpha_1\mu_1)\epsilon_2-\mu\epsilon_1\epsilon_2$
and $D=D_0+\beta\epsilon_1\epsilon_2\,x_1$.

\noindent\scalebox{0.83}{$\displaystyle
S_1=\frac{\beta_1\mu_2\beta_w
         -\mu_1\beta_w(\alpha_1\mu_2-\alpha_2\mu_2+\beta_2+(\beta x_1-\mu)\epsilon_2)
         -\mu\epsilon_2(\beta_1-\alpha_1\mu_1)(\mu-\beta x_1)}
        {\gamma_1\beta_w D},
$}

\noindent\scalebox{0.83}{$\displaystyle
S_2=\frac{\beta_2\mu_1\beta_w
         -\mu_2\beta_w(-\alpha_1\mu_1+\alpha_2\mu_1+\beta_1+(\beta x_1-\mu)\epsilon_1)
         -\mu\epsilon_1(\beta_2-\alpha_2\mu_2)(\mu-\beta x_1)}
        {\gamma_2\beta_w D};
$}

\noindent $x_1$ satisfies $A_{\mathrm{EE}}x_1^2+B_{\mathrm{EE}}x_1+C_{\mathrm{EE}}=0$ with
\begin{align*}
A_{\mathrm{EE}}&=-\beta\mu\epsilon_1\epsilon_2(\beta+\beta_w),\\
B_{\mathrm{EE}}&=-\beta_w\mu D_0
               +\beta\epsilon_1\epsilon_2(\beta_w\Lambda+\mu^2)
               -\beta\beta_w(\mu_2\epsilon_1+\mu_1\epsilon_2),\\
C_{\mathrm{EE}}&=\beta_w\Lambda D_0.
\end{align*}

\begin{table}[H]
  \centering
  \small
  \caption{Equilibria of the OSN model ($\omega>0$): invaders, residents, existence and
  extra stability conditions.  Relay pairs may be read from right to left and downwards.
  Conditions marked $\dagger$ are predicted by the relay structure
  (Theorem~\ref{thm:relay-abstract}); those for OSND, gOSN, RFE are verified directly.
  Compared with Table~\ref{tab:OSN_equilibria_om0}, the gOSN-branch endemic equilibria
  $E_{1g}$, $E_{2g}$, $\mathrm{EE}_g$ are absent (they collapse into $E_1$, $E_2$, EE
  via the $R\to W$ coupling), and $W$, $R$ become residents at all strain-endemic rows.}
  \label{tab:OSN_equilibria_om}
  \begin{tabular}{|c|c|c|c|c|}
  \hline
  \textbf{Name} & \textbf{Invaders} & \textbf{Residents} & \textbf{Existence} & \textbf{Extra Stability} \\
  \hline\hline
  OSND & $U,W,S_1,B_1,S_2,B_2,R$ & $x_1$ & always & $R_0<1$ \\
  \hline
  gOSN & $W,S_1,B_1,S_2,B_2,R$ & $x_1,U$ & $R_0>1$
       & $\max_j R_j(\text{gOSN})<1$,\ $R_0<1+\fr{\beta}{\beta_w}$ \\
  \hline
  $E_1$ & $S_2,B_2$ & $x_1,U,W,S_1,B_1,R$ & $R_1(\text{gOSN})>1$
        & $R_2(\text{gOSN})<1$\,$^\dagger$ \\
  \hline
  $E_2$ & $S_1,B_1$ & $x_1,U,W,S_2,B_2,R$ & $R_2(\text{gOSN})>1$
        & $R_1(\text{gOSN})<1$\,$^\dagger$ \\
  \hline
  EE & $\emptyset$ & $x_1,U,W,S_1,B_1,S_2,B_2,R$ & $\max_j R_j(\text{gOSN})>1$\,$^\dagger$
     & always$^\dagger$ \\
  \hline
  RFE & $S_1,B_1,S_2,B_2,R$ & $x_1,U,W$ & $R_0>1+\fr{\beta}{\beta_w}$
      & $\max_j R_j(\text{RFE})<1$ \\
  \hline
  \end{tabular}
\end{table}

\beR
Every threshold in the Extra Stability column of a row also appears in the Existence
column of the next relay-successor, confirming the relay-pair structure:
losing stability at one equilibrium coincides with the appearance of the next.
The gOSN branch (upper part of Table~\ref{tab:OSN_equilibria_om}) is active when
$R_0 < 1 + \fr{\beta}{\beta_w}$ (equivalently $\hat x_1 > \tilde x_1$, i.e.\ RFE does not
exist), and the RFE branch when $R_0 > 1+\fr{\beta}{\beta_w}$.
Relay pairs from RFE (if $R_j(\text{RFE})>1$) are not yet fully analysed; they would
add further rows above RFE analogous to Table~\ref{tab:OSN_equilibria_om0}.
\eeR

%------------------------------------------------------------------
\beT[Relay structure of the OSN model when omega is positive]\label{thm:relay-omP}
The following hold for all $\omega\ge 0$.
\begin{enumerate}[label=(\roman*)]
\item \emph{(Block factorization.)} The block factorization~\eqref{eq:OSN-block-factor}
  and the invasion numbers $R_j(\mathrm{gOSN})$, $R_j(\mathrm{RFE})$ are independent
  of $\omega$: neither $\omega$, $R$, nor $W$ appears in the strain-block
  equations~\eqref{eq:OSN-invasion}, so the transversal blocks $M_j$ at gOSN and RFE
  are the same as for $\omega=0$.
\item \emph{(Relay $\mathrm{gOSN}\to E_j$, $j=1,2$: proved.)} The boundary transcritical
  argument of Theorem~\ref{thm:relay}(a) applies verbatim: $R$ and $W$ are decoupled
  from $(S_j,B_j)$ at gOSN (since $R^*=W^*=0$ there), so the zero null-vector
  $v_j=(\mu_j,\gamma_j)^\top\gg0$ gives the bifurcation into $\{S_j,B_j>0\}$.
  Once the branch appears, \eqref{eq:RW-link} forces $R,W>0$ immediately.
  The existence condition is $R_j(\mathrm{gOSN})>1$, same as for $\omega=0$.
\item \emph{(gOSN unstable toward $E_j$ and EE: partly proved;
  existence of EE and LAS of $E_j$, EE: predicted~$\dagger$.)}
  When $\max_j R_j(\mathrm{gOSN})>1$, item~(ii) shows that gOSN is transversally
  unstable in both strain directions, so the relay framework predicts sequential
  transitions $\mathrm{gOSN}\to E_j\to\mathrm{EE}$.

  \smallskip\noindent\emph{What is proved:}
  \begin{itemize}
  \item gOSN is transversally unstable in both strain directions (item~(ii) above).
  \item At any $E_j$ on the gOSN branch, $W$ is already a resident
    (forced by~\eqref{eq:RW-link}), and no new $W$-invasion threshold arises:
    the gOSN-branch assumption $R_0^W<1$ gives $U^*_{E_j}\le\hat U<\mu/\beta_w$,
    hence $\partial_W\dot W|_{E_j}=\beta_w U^*_{E_j}-\mu<0$ automatically,
    without the extra condition $R_0<1+\beta/\beta_w$ that was needed in
    $E_{1g},E_{2g}$ for $\omega=0$.
  \end{itemize}

  \smallskip\noindent\emph{What is predicted~($\dagger$), pending direct
  Routh--Hurwitz verification:}
  \begin{itemize}
  \item \emph{Stability of $E_j$} under $R_{3-j}(\mathrm{gOSN})<1$:
    the characteristic polynomial at $E_j$ involves irrational coordinates
    (Section~\ref{s:coordsomP}); the full Hurwitzness of the tangential block
    has not been verified.
  \item \emph{Existence of EE} when $\max_j R_j(\mathrm{gOSN})>1$:
    the quadratic for $x_1$ in Section~\ref{s:coordsomP} is expected to yield
    a positive root with $S_1,S_2>0$, but the exact parameter domain for
    this has not been determined.
  \item \emph{Stability of EE} (listed as ``always'' in
    Table~\ref{tab:OSN_equilibria_om}): predicted by the relay structure
    (EE has no absent siphon block to invade it) but not directly verified.
  \end{itemize}
\item \emph{(Relay $\mathrm{gOSN}\to\mathrm{RFE}$: proved.)}
  On the strain-free face $\{S_j=B_j=0\}$ one has $\dot R=-\omega R$, forcing $R^*=0$;
  the $(W,R)$-block at gOSN is upper-triangular with spectral abscissa
  $\beta_w U^*-\mu=\mu(R_0^W-1)$ (cf.\ Theorem~\ref{thm:relay} Remark).
  This relay is governed by $R_0^W$ for all $\omega\ge0$.
\item \emph{(Relays from RFE: deferred.)}
  Relays $\mathrm{RFE}\to E_j$ (if $R_j(\mathrm{RFE})>1$) are analogous to~(ii)--(iii)
  with gOSN replaced by RFE. The strain-block argument carries over; the full
  Routh--Hurwitz verification of stability at $E_j$ in this branch is left for future work.
\end{enumerate}
\eeT

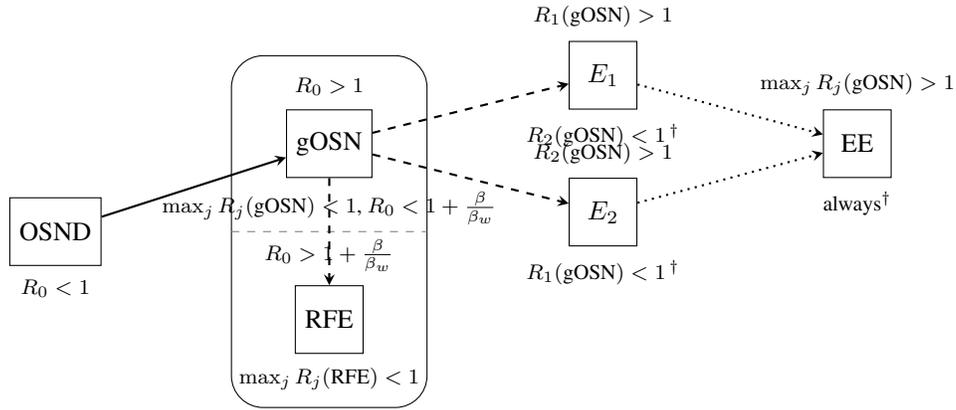
\begin{figure}[h!]
\centering
\begin{tikzpicture}[scale=1.3,
    node distance=2cm,
    box/.style={rectangle, draw, minimum size=0.9cm, font=\small},
    lbl/.style={font=\scriptsize, align=center},
    arrow/.style={->, thick, >=stealth},
    darrow/.style={->, thick, >=stealth, dashed},
    dotarrow/.style={->, thick, >=stealth, dotted}
]
% Nodes
\node[box] (OSND) at (0,   0)   {OSND};
\node[box] (gOSN) at (2.8, 0.9) {gOSN};
\node[box] (RFE)  at (2.8,-0.9) {RFE};
\node[box] (E1)   at (5.6, 1.6) {$E_1$};
\node[box] (E2)   at (5.6, 0.2) {$E_2$};
\node[box] (EE)   at (8.2, 0.9) {EE};

% Grouping boxes
\draw[rounded corners=12pt] (1.8,-1.8) rectangle (3.8, 1.8);
\draw[gray, dashed] (1.8,0) -- (3.8,0);

% Existence labels (above nodes)
\node[lbl, above=0.05cm of gOSN] {$R_0>1$};
\node[lbl, above=0.05cm of RFE]  {$R_0>1+\fr{\beta}{\beta_w}$};
\node[lbl, above=0.05cm of E1]   {$R_1(\text{gOSN})>1$};
\node[lbl, above=0.05cm of E2]   {$R_2(\text{gOSN})>1$};
\node[lbl, above=0.05cm of EE]   {$\max_j R_j(\text{gOSN})>1$};

% Stability labels (below nodes)
\node[lbl, below=0.05cm of OSND] {$R_0<1$};
\node[lbl, below=0.05cm of gOSN] {$\max_j R_j(\text{gOSN})<1$,\ $R_0<1+\fr{\beta}{\beta_w}$};
\node[lbl, below=0.05cm of RFE]  {$\max_j R_j(\text{RFE})<1$};
\node[lbl, below=0.05cm of E1]   {$R_2(\text{gOSN})<1$\,$^\dagger$};
\node[lbl, below=0.05cm of E2]   {$R_1(\text{gOSN})<1$\,$^\dagger$};
\node[lbl, below=0.05cm of EE]   {always$^\dagger$};

% Arrows
\draw[arrow]  (OSND) -- (gOSN);
\draw[darrow] (gOSN) -- (RFE);
\draw[darrow] (gOSN) -- (E1);
\draw[darrow] (gOSN) -- (E2);
\draw[dotarrow] (E1) -- (EE);
\draw[dotarrow] (E2) -- (EE);
\end{tikzpicture}
\caption{Relay order between the six equilibria ($\omega>0$).
Existence conditions appear above each node and extra stability conditions below.
Solid arrow (OSND$\to$gOSN): full relay.
Dashed arrows: multiple relay s (existence at head $\Rightarrow$ instability at tail).
Dotted arrows: predicted multiple relay s (pending direct verification, marked $\dagger$).
The gOSN/RFE split (dashed separator) is governed by $R_0^W=1$.
Unlike the $\omega=0$ case (Figure~\ref{fig:relay0}), there are no gOSN-branch endemic
equilibria with $W=0$: the $R\to W$ coupling forces $W>0$ at $E_1$, $E_2$, EE.
Relays from RFE are analogous but not yet fully analysed.}
\label{fig:relay}
\end{figure}

%------------------------------------------------------------------
\ssec{Direct LAS analysis for the rational equilibria}\lbl{s:LAS}

We study the LAS of gOSN and RFE directly via the Routh--Hurwitz conditions applied
to the characteristic polynomial at each rational equilibrium (script \texttt{OSN.wl},
function \texttt{staPP}).

\begin{enumerate}

\item \textbf{gOSN}: $S_1=S_2=B_1=B_2=W=R=0$,
  $\hat x_1 = \frac{\mu_n}{\beta}$, $\hat U = \frac{\mu(R_0-1)}{\beta}$.
  gOSN exists precisely when $R_0>1$.

  The characteristic polynomial at gOSN factors into two linear factors and three
  quadratic ones.  The stability condition from the non-trivial linear factor is
  $\hat x_1 > \tilde x_1$, i.e.\
  \begin{equation}\label{eq:gOSN-W}
    R_0 \;<\; 1 + \frac{\beta}{\beta_w}
    \quad\Bigl(\Leftrightarrow\; R_0^W < 1\Bigr),
  \end{equation}
  which is equivalent to RFE not existing.
  The stability conditions from the two non-trivial quadratic factors are:
  \begin{equation}\label{eq:gOSN}
    1 \;>\; R_{\mathrm{gOSN}}
    \;=\; \max\!\Bigl[R_1(\mathrm{gOSN}),\,R_2(\mathrm{gOSN})\Bigr],
  \end{equation}
  where $R_j(\mathrm{gOSN}) = \frac{\beta_j \hat U}{\mu_j(1+\alpha_j \hat U)}$
  are the invasion numbers of the two rumors at gOSN.
  Thus gOSN is LAS iff both \eqref{eq:gOSN-W} and \eqref{eq:gOSN} hold.

  \beR
  Note that $\hat x_1 < \tilde x_1$ (RFE exists) implies gOSN cannot be LAS,
  since one of the linear eigenvalues becomes positive.
  The threshold $\hat x_1 = \tilde x_1$ can also be written as $R_0^W = 1$, or
  equivalently as $R(\mathrm{RFE}) = 1$ (the reproduction number at RFE equals one).
  \eeR

\item \textbf{RFE}: $S_1=S_2=B_1=B_2=R=0$, $W>0$,
  \begin{equation}\label{eq:x1}
    \tilde x_1 \;=\; x_d\,\frac{\beta_w}{\beta_w+\beta},
    \qquad
    \tilde U \;=\; \frac{\mu}{\beta_w},
    \qquad
    \tilde W \;=\; \frac{\beta(\tilde x_1 - \hat x_1)}{\beta_w} > 0
    \;\Leftrightarrow\;
    R_0 > 1 + \frac{\beta}{\beta_w}.
  \end{equation}
  The characteristic polynomial at RFE factors into a trivial linear factor,
  a cubic factor that is always Hurwitz when $R_0>1$ (verified directly), and
  two quadratic factors whose stability requires:
  \begin{equation}\label{eq:RFE}
    1 \;>\; R_{\mathrm{RFE}}
    \;=\; \max\!\Bigl[R_1(\mathrm{RFE}),\,R_2(\mathrm{RFE})\Bigr],
  \end{equation}
  where $R_j(\mathrm{RFE}) = \frac{\beta_j \tilde U}{\mu_j(1+\alpha_j \tilde U)}$
  are the invasion numbers of the two rumors at RFE.

\item \textbf{$E_1$, $E_2$, EE}: these equilibria carry $W,R>0$ and are irrational
  (they depend on $f_j$).  At $E_1$, the characteristic polynomial factors into a
  quadratic and a degree-6 polynomial; a complete Routh--Hurwitz analysis
  is deferred to future work.  Table~\ref{tab:OSN_equilibria_om} lists the stability
  conditions predicted by Theorem~\ref{thm:relay-abstract}.

\end{enumerate}

\ssec{Alternative derivation of fixed-point coordinates}\lbl{s:altcoords}

\noindent\emph{(This subsection gives the detailed algebraic derivation underlying
Section~\ref{s:coordsomP}; it may be skipped on first reading.)}

Finding RUR polynomial equations requires first choosing the ``keepVar'' variable to be
kept, and the equation which will be dropped.  In \texttt{clsEq} we chose the ad-hoc
solution of dropping the first non-DFE variable and its equation.  However,
the existence problem requires further human intervention, because the RUR equilibria
might require also that the polynomial equations admit roots in certain intervals.
Due to these complications, we turn now to a detailed direct derivation.
First, the factoring  of the second equation in the fixed point system
$$\bff 0=\left(
\begin{array}{c}
 \La -\mu x_1-\alpha x_1 x_2 \\
 -x_2 \left(\mu -\beta  x_1+\beta _w W\right) \\
 \beta _w x_2 W-\mu  W+R\omega  \\
 \frac{\beta _1 B_1 x_2}{B_1 \epsilon _1+\alpha _1 x_2+1}-\gamma _1 S_1 \\
 \frac{\beta _2 B_2 x_2}{B_2 \epsilon _2+\alpha _2 x_2+1}-\gamma _2 S_2 \\
 \gamma _1 S_1-B_1 \mu _1 \\
 \gamma _2 S_2-B_2 \mu _2 \\
 B_1 \mu _1+B_2 \mu _2-\omega R \\
\end{array}
\right)$$   leads on one hand  to OSND, and on the other hand to the  system with second equation simplified by
$x_2$:
$$\bff 0=\left(
\begin{array}{c}
 \La -\mu x_1-\beta x_1 x_2 \\
 \mu - \beta x_1+\beta _w W \\
 \beta _w x_2 W-\mu  W+R\omega  \\
 \frac{\beta _1 B_1 x_2}{B_1 \epsilon _1+\alpha _1 x_2+1}-\gamma _1 S_1 \\
 \frac{\beta _2 B_2 x_2}{B_2 \epsilon _2+\alpha _2 x_2+1}-\gamma _2 S_2 \\
 \gamma _1 S_1-B_1 \mu _1 \\
 \gamma _2 S_2-B_2 \mu _2 \\
 B_1 \mu _1+B_2 \mu _2-\omega R \\
\end{array}
\right),$$
where equations $2,6,7,8$ are linear.
\BEN

\im
  We may eliminate now from equations $2,3,6,7,8$ the corresponding variables
$x_2,W,B_1,B_2,R$,  in function of $x_1,S_1,S_2 $. This yields the \rat:
\BEN \im
\be{el} \bep
%x_2\to \frac{  +\mu ^2-\alpha  \mu  x_1+\gamma _1 \beta _w S_1+\gamma _2 \beta _w S_2}{\beta _w \left(\mu -\beta x_1\right)}
W\to \frac{\beta  x_1-\mu _n}{\beta _w}=\frac{\beta  (x_1- \H x_1)}{\beta _w}
\\B_1\to \frac{\gamma _1 S_1}{\mu_1}
\\B_2\to \frac{\gamma _2 S_2}{\mu_2}
\\R\to \frac{\gamma _1 S_1+\gamma _2 S_2}{\omega }.\eep
\ee
\im  \be{x2} x_2\to \frac{\mu }{\beta _w}+\frac{\gamma _1 S_1+\gamma _2 S_2}{\mu _n-\beta  x_1}\ee
\EEN
 The positivity of
$B_1,B_2,R$ holds whenever $S_1,S_2$ are positive, that
 of $W$ requires also that $   x_1- \H x_1 >0$, and finally that
 of $x_2$ imposes also
 \be{x2p}   x_1- \H x_1 >x_c:=\fr{\beta _w(\ga_1 S_1 + \ga_2 S_2)}{\mu \beta}.\ee

\im A pleasant surprise  is that after eliminating  the variables
$ W,B_1,B_2,R$, the equations for $S_1,S_2$ factor:
\be{45} \bep
\gamma _1 S_1 \left(\frac{\beta _1 x_2}{\mu_1+\alpha _1 \mu_1 x_2+\gamma _1 S_1 \epsilon _1}-1\right)=0\\\gamma _2 S_2 \left(\frac{\beta _2 x_2}{\mu_2+\alpha _2 \mu_2 x_2+\gamma _2 S_2 \epsilon _2}-1\right)=0
\eep
\ee

 After eliminating also $x_2$,  the four  cases yield \resp:
 \be{4c} \bc (S_1,S_2)\to (0,0) &\text{RFE,gOSN}\\
 (S_1,S_2)\to (\frac{\mu(\mR  x_1  -1) \left(\beta _1 \mu -\mu_1 \left(\alpha _1 \mu +\beta _w\right)\right)}{\gamma _1 \beta _w \left(-\alpha _1 \mu_1+\beta _1+\beta x_1 \epsilon _1-\mu  \epsilon _1\right)}, 0) &\text{E1}\\
  (S_1,S_2)\to (0,\frac{\mu(\mR  x_1  -1) \left(\beta _2 \mu -\mu_2 \left(\alpha _2 \mu +\beta _w\right)\right)}{\gamma _2 \beta _w \left(-\alpha _2 \mu_2+\beta _2+\beta x_1 \epsilon _2-\mu  \epsilon _2\right)})&\text{E2} \\
 (S_1,S_2)\to (\frac{\beta _1 \mu_2 \beta _w-\left(\mu_1 \beta _w \left(\alpha _1 \mu_2-\alpha _2 \mu_2+\beta _2+\beta x_1 \epsilon _2-\mu  \epsilon _2\right)\right)+\mu  \left(-\epsilon _2\right) \left(\beta _1-\alpha _1 \mu_1\right) \left(\mu -\beta x_1\right)}{\gamma _1 \beta _w \left(-\alpha _2 \mu_2 \epsilon _1+\beta _2 \epsilon _1+\epsilon _2 \left(-\alpha _1 \mu_1+\beta _1+\beta x_1 \epsilon _1-\mu  \epsilon _1\right)\right)},\\ \frac{\beta _2 \mu_1 \beta _w-\mu_2 \beta _w \left(-\alpha _1 \mu_1+\alpha _2 \mu_1+\beta _1+\beta x_1 \epsilon _1-\mu  \epsilon _1\right)+\mu  \left(-\epsilon _1\right) \left(\beta _2-\alpha _2 \mu_2\right) \left(\mu -\beta x_1\right)}{\gamma _2 \beta _w \left(-\alpha _2 \mu_2 \epsilon _1+\beta _2 \epsilon _1+\epsilon _2 \left(-\alpha _1 \mu_1+\beta _1+\beta x_1 \epsilon _1-\mu  \epsilon _1\right)\right)}) &\text{EE}\ec
\ee

\im   Eliminating  $ S_1,S_2$, in each case, and plugging
 in the first equation leads to the following
equations for $x_1$:
\be{x1} \bc \f(x_1) -
\frac{\beta  \mu  x_1 \left( \mu -\beta x_1\right)}{\beta _w \left(\mu -\beta x_1\right)}
=\Lambda -\mu  x_1(1+\frac{\beta   }{\beta _w })=0\\
\f(x_1)+\frac{\beta x_1 \left(\mu  \epsilon _1 \left( \mu -\beta x_1\right)-\mu_1 \beta _w\right)}{\beta _w \left(-\alpha _1 \mu_1+\beta _1+\beta x_1 \epsilon _1-\mu  \epsilon _1\right)}=0
\\
\f(x_1)+\frac{\beta x_1 \left(\mu  \epsilon _2 \left( \mu -\beta x_1\right)-\mu_2 \beta _w\right)}{\beta _w \left(-\alpha _2 \mu_2+\beta _2+\beta x_1 \epsilon _2-\mu  \epsilon _2\right)}=0
\\
\f(x_1)+\frac{\beta x_1 \left(\mu  \epsilon _1 \epsilon _2 \left( \mu -
\beta x_1\right)-\beta _w \left(\mu_2 \epsilon _1+\mu_1 \epsilon _2\right)\right)}{\beta _w \left(-\alpha _2 \mu_2 \epsilon _1+\beta _2 \epsilon _1+\epsilon _2 \left(-\alpha _1 \mu_1+\beta _1+\beta x_1 \epsilon _1-\mu  \epsilon _1\right)\right)}=0
\ec
\ee
for the four cases in \eqr{4c}, \resp. When $\f(x_1)=\Lambda -\mu  x_1$, the
numerators are of degree $1,2,2,2$
in $x_1$.
\EEN
\ssec{Existence conditions for $E_1$ and $E_2$}

\noindent\emph{(This subsection gives the detailed verification that $R_j(\mathrm{gOSN})>1$ is
the exact existence condition; it may be skipped on first reading.)}

 The second  case in \eqr{4c} shows that at $E_1$   we need also $   x_1- \H x_1 >0$, so
 we perform the substitution $   x_1=y + \H x_1 .$ The coefficients of the resulting
  equation $A y^2 + B y + C$:

 $$\bc C=  -\beta _w \left(-\alpha _1 \beta  \Lambda  \mu _1+\beta  \beta _1 \Lambda +\alpha _1 \mu  \mu _1 \mu _n-\beta _1 \mu  \mu _n-\beta  \mu _1 \mu _n\right)\\
 B=\beta  \left(\beta  \beta _w \left(\mu _1-\Lambda  \epsilon _1\right)+\beta _w \mu  \left(-\alpha _1 \mu _1+\beta _1+\epsilon _1 \mu _n\right)+\beta  \mu  \epsilon _1 \mu _n\right)
  \\A=\beta ^2 \left(\beta +\beta _w\right) \mu  \epsilon _1
   \ec$$
  and there is exactly one \nne\ root,
 provided that  $$(\xd -\H  x_1)(\beta _1-\alpha _1 \mu _1)>   \mu _1 \fr{\mu_n}{ \mu
   } \Lra $$   \be{mys}(\beta _1-\alpha _1 \mu_1)\fr{ \mu
   }{\mu_n}(x_0 -\H x_1)=(\beta _1-\alpha _1 \mu_1)\H x_2>\mu _1 \Eq
   R_1(gOSN) >1
 \ee
 (recall that $\hat{x}_2=\fr{ \mu
   }{\mu_n}(x_0 -\H x_1)$).

 Now the second inequality for the positivity of $x_2$,
 $$x_2=\frac{\beta  \left(x_1-\hat{x}_1\right) \left(\mu _1 \tilde{x}_1+\mu
  \left(x_1-\hat{x}_1\right) \epsilon _1 \left(x_0-\tilde{x}_1\right)\right)}
  {\mu  \left(x_0-\tilde{x}_1\right)
 \left(\beta _1-\alpha _1 \mu _1+\beta  \left(x_1-\hat{x}_1\right) \epsilon _1\right)}
 >0$$
 may be shown to hold always, so we may conclude that the existence domain of $E_1$ is precisely $R_1(gOSN) >1,$  and hence E1 is a second order multiple relay  of gOSN,
 with a   similar conclusion for E2.
 \iffalse
 split in two cases, according to the sign of $x_0-\tilde{x}_1$.
 When this is positive, the inequality is already implied by previous conditions,
 but in the
 opposite case we must have also
 $$\left(x_1-\hat{x}_1\right)>\frac{\mu _1 \tilde{x}_1}
 { \mu \epsilon _1 \left(\tilde{x}_1-x_0\right)},
 $$
so a new shift $   x_1=y + \H x_1 +\frac{\mu _1 }
 { \mu \epsilon _1 \left(1-x_0/\tilde{x}_1\right)}$ is required.

\fi

%

\section{Can the $\omega=0$ model oscillate? An impossibility result}
\lbl{s:Hopf}

The standard route to establishing (or ruling out) periodic solutions for a positive ODE proceeds
in three steps.
\BEN
\im \textbf{Structural screening.} Use the Cauchy--Binet oscillation recipes (Recipe~I and Recipe~II
of \cite{BSV,AABH25}) implemented in \texttt{EpidCRN} as \texttt{oscRI} and \texttt{oscRII}.
These give purely structural \emph{necessary} conditions for Hopf-type instability; if neither
recipe is triggered, periodic orbits are structurally impossible.

\im \textbf{Local analysis.} For each equilibrium $E^*$, study the characteristic polynomial
$p(\lambda)=\det(J(E^*)-\lambda I)$; determine whether it can admit a pair of purely imaginary
roots $\pm i\omega$ as a parameter is varied.

\im \textbf{Lyapunov-coefficient computation.}  If the previous step yields a candidate Hopf
point, compute the first Lyapunov coefficient $\ell_1$ to determine supercritical ($\ell_1<0$,
stable limit cycle) or subcritical ($\ell_1>0$, unstable limit cycle) character.
\EEN

We carry out steps 1 and 2 below and find that Hopf bifurcation is \emph{structurally impossible}
for the $\omega=0$ model, so step 3 is moot.

\subsubsection*{Block lower-triangular structure of the Jacobian}

The $\omega=0$ model is the 7-dimensional reduction of \eqr{AOSN} obtained by dropping $R$
(since $R'=\mu_1 B_1+\mu_2 B_2\ge 0$ and $R$ does not appear in any other equation when $\omega=0$).

The key structural observation is that $S_1,S_2,B_1,B_2$ do \emph{not appear} in the equations
for $x_1,U,W$: the rumor-believers can only \emph{consume} new users at rate $f_j$, but in this
model the catalytic reaction
$$B_j + U \;\to\; B_j + U + S_j$$
leaves both $U$ and $B_j$ \emph{unchanged} (net stoichiometry zero for those species).
Consequently, the platform dynamics $(x_1,U,W)$ are autonomous, and the strains are
``slave'' variables driven by $U(t)$.

\beT[Three blocks lower-triangular structure when omega=0]\label{t:blt}
The Jacobian $J(x)$ of the $\omega=0$ reduction of \eqr{AOSN} has block lower-triangular form
\be{blt}
J =
\begin{pmatrix}
J_P & 0 & 0 \\
*   & J_{1} & 0 \\
*   & 0 & J_{2}
\end{pmatrix}
\ee
at every point $x\in\mathbb{R}^7_{\ge 0}$, where
$J_P = D_{(x_1,U,W)}(x_1',U',W')$ is the $3\times3$ platform block, and
$$J_j = \begin{pmatrix}-\gamma_j & \partial_{B_j}f_j \\ \gamma_j & -\mu_j\end{pmatrix}$$
is the $2\times2$ strain-$j$ block ($j=1,2$).
\eeT

\begin{proof}
Inspecting \eqr{AOSN} with $\omega=0$: the right-hand sides of the $x_1',U',W'$ equations are independent of
$S_1,S_2,B_1,B_2$.  Hence the corresponding columns of $J$ are zero in the first three rows,
giving the zero upper-right $3\times4$ block in \eqr{blt}.  The independence of strain-1
variables from strain-2 variables and vice versa gives the zero off-diagonal $2\times2$ blocks
between $J_1$ and $J_2$.
\end{proof}

\beR
The eigenvalues of $J(E^*)$ at any equilibrium $E^*$ are the \emph{union} of the spectra of
$J_P(E^*)$, $J_1(E^*)$, and $J_2(E^*)$.  It therefore suffices to analyse each $3\times3$ (or
$2\times2$) block separately.
\eeR

\subsubsection*{Analysis of the strain blocks}

\beP[No Hopf from strain blocks]\label{p:noHopf_str}
The characteristic polynomial of $J_j$ is
$$p_j(\lambda) = \lambda^2 + (\gamma_j+\mu_j)\,\lambda + \gamma_j\bigl(\mu_j - \partial_{B_j}f_j\bigr).$$
Since $\gamma_j,\mu_j>0$, the trace equals $-(\gamma_j+\mu_j)<0$ everywhere.  A pair of purely
imaginary eigenvalues $\pm i\omega$ requires $\operatorname{tr}(J_j)=0$, which is impossible.
Hence $J_j$ can never have purely imaginary eigenvalues; in particular, Hopf bifurcation cannot
originate in the strain variables.
\eeP

\subsubsection*{Analysis of the platform block}

The platform block $J_P$ takes two forms depending on whether $W=0$ or $W>0$.

\medskip\noindent
\textbf{Case $W=0$ (gOSN branch).}  At any equilibrium with $W=0$, the $W$-equation decouples
and $J_P$ reduces to the $2\times 2$ matrix
$$J_P\big|_{W=0} = \begin{pmatrix}-\mu-\beta U & -\beta x_1\\ \beta U & \beta x_1-\mu_n\end{pmatrix}.$$
At an equilibrium on the gOSN branch, $\beta x_1 = \mu_n$, so the $(2,2)$ entry vanishes and
$$J_P\big|_{\text{gOSN}} = \begin{pmatrix}-\mu-\beta\Hat U & -\beta\Hat x_1\\ \beta\Hat U & 0\end{pmatrix},$$
with characteristic polynomial $\lambda^2+(\mu+\beta\Hat U)\lambda+\beta^2\Hat x_1\Hat U$.
Both coefficients are strictly positive (since $\Hat x_1,\Hat U>0$), so
$\operatorname{tr}(J_P)=-(\mu+\beta\Hat U)<0$ and $\det(J_P)=\beta^2\Hat x_1\Hat U>0$.
Purely imaginary eigenvalues would require trace zero: impossible.

\medskip\noindent
\textbf{Case $W>0$ (RFE branch).}  At any equilibrium with $W>0$, the equilibrium conditions
$\beta_w U = \mu$ (from $W'=0$) and $\beta x_1 = \mu_n + \beta_w W$ (from $U'=0$) give
$$J_P\big|_{\text{RFE}} = \begin{pmatrix}-\mu-\beta U & -\beta x_1 & 0\\ \beta U & 0 & -\beta_w U\\ 0 & \beta_w W & 0\end{pmatrix}.$$
The characteristic polynomial is
\be{pP}
p_P(\lambda) = \lambda^3 + (\mu+\beta U)\,\lambda^2
   + \bigl(\beta_w^2 U W + \beta^2 x_1 U\bigr)\,\lambda
   + (\mu+\beta U)\beta_w^2 U W.
\ee
\beP[Hurwitz stability of $J_P$ on the RFE branch]\label{p:hur3}
All coefficients in \eqr{pP} are strictly positive, and the Routh--Hurwitz criterion for
degree $3$, namely $a_1 a_2 > a_3$, holds:
$$(\mu+\beta U)\bigl(\beta_w^2 U W+\beta^2 x_1 U\bigr) > (\mu+\beta U)\beta_w^2 U W,$$
since $\beta^2 x_1 U>0$.  Hence all three eigenvalues of $J_P$ have strictly negative real
parts; purely imaginary eigenvalues are impossible.
\eeP

\subsubsection*{Main result}

\beT[Hopf impossibility when omega vanishes]\label{t:noHopf}
The $\omega=0$ reduction of \eqr{AOSN} cannot undergo a Hopf bifurcation at any equilibrium
$E^*\in\mathbb{R}^7_{\ge 0}$ and for any positive parameter values.
\eeT

\begin{proof}
By Theorem~\ref{t:blt}, the eigenvalues of $J(E^*)$ are the union of the spectra of $J_P(E^*)$,
$J_1(E^*)$, $J_2(E^*)$.  A necessary condition for Hopf bifurcation is the existence of a pair
of purely imaginary eigenvalues of $J(E^*)$.  Proposition~\ref{p:noHopf_str} rules this out for
$J_1$ and $J_2$ (their traces are always strictly negative).  The $W=0$ case rules it out for
$J_P$ by the same trace argument, and the $W>0$ case is covered by
Proposition~\ref{p:hur3}.
\end{proof}

\beR
Theorem~\ref{t:noHopf} is a purely \emph{structural} result: it does not depend on specific
parameter values, only on the sign structure of \eqr{AOSN} with $\omega=0$.  The key mechanism is
the absence of feedback from the strain compartments $(S_j,B_j)$ to the platform compartments
$(x_1,U,W)$: rumor spread cannot destabilize the platform dynamics in an oscillatory fashion.
This is in contrast to standard SEIR-type models where infectives reduce the susceptible pool
and feedback is present.
\eeR

\beR[Connection to the BSV oscillation recipes]
The structural impossibility established above is also reflected at the level of the Cauchy--Binet
oscillation recipes.  Recipe~I \cite{BSV} requires a negative oriented cycle in the species
graph of the linearisation; Recipe~II requires a suitable Jacobian minor to change sign.
Both are preempted by the block lower-triangular structure \eqr{blt}: no negative cycle
can cross the boundary between the platform and strain sub-networks.
The EpidCRN commands \texttt{oscRI[RN,rts,va7]} and \texttt{oscRII[RN,rts,va7]}  (where
\texttt{va7} is the 7-dim variable list) should return empty core sets, confirming
Theorem~\ref{t:noHopf} computationally.
\eeR

\subsubsection*{What modifications could produce oscillations?}

Theorem~\ref{t:noHopf} identifies the specific structural feature that rules out Hopf: the
one-way coupling from platform to strains.  Oscillations could in principle be generated by any
of the following modifications:
\BEN
\im \textbf{Feedback from believers to the platform.} If a large number of believers $B_j$
accelerates the departure from the platform (\ie\ increases the rate $U\to W$), the
coupling becomes bidirectional and Hopf becomes structurally possible.  Concretely, replacing
the abandonment rate $\beta_w U W$ by $(\beta_w W + \delta B_1 + \delta B_2) U$ introduces
non-zero $(2,6)$ and $(2,7)$ entries in $J$, breaking the block structure.

\im \textbf{Re-infection / rumor revival.} If former skeptics $R$ can re-enter the susceptible
pool $S_j$ (a form of waning immunity to rumor), the $R$ compartment is coupled back to $S_j$,
potentially creating a negative feedback loop long enough to sustain oscillations.

\im \textbf{Non-zero $\omega$.}  Restoring $\omega>0$ couples $R$ back to $W$, which then
acts on $U$; the platform dynamics become 3-dimensional with a feedback loop
$U\to B_j\to R\to W\to U$, and the Jacobian is no longer block lower-triangular.
This is why the $\omega>0$ model is a natural candidate for oscillatory behavior, and
the Hopf analysis there remains open.
\EEN
\subsection{Beyond transcritical relays: Hopf and Bogdanov--Takens relays}

Up to this point, all relays considered in this paper are \emph{transcritical relays}:
the loss of stability of a resident equilibrium on a siphon face is caused by a
simple zero eigenvalue of a transversal Jacobian block, and the successor equilibrium
appears on an adjacent face.
This mechanism is generic for ME--type siphons, since transversal blocks are Metzler
and their dominant eigenvalue is real.

It is natural to ask whether other bifurcation types can generate relays along the
siphon lattice.
In this subsection we introduce two further notions---Hopf relays and
Bogdanov--Takens relays---and clarify precisely which structural properties must fail
for them to occur.

\subsubsection{Hopf relays}

\beD[Hopf relay]\label{def:hopf-relay}
Let $(\Sigma',\Sigma)$ be a distance-one cover in the siphon lattice
(Theorem~\ref{thm:relay-abstract}) with invader block $\sigma=\Sigma\setminus\Sigma'$,
and let $E^*\in\mathcal{F}_\Sigma$ be an equilibrium that is asymptotically stable
relative to $\mathcal{F}_\Sigma$.
A \emph{Hopf relay} along this cover occurs if, as a parameter $\theta$ varies,
\begin{itemize}
\item the transversal Jacobian block $M_\sigma(E^*,\theta)$
  (Definition~\ref{d:transblock}) has a simple pair of
  complex conjugate eigenvalues crossing the imaginary axis with nonzero speed;
\item all remaining eigenvalues of $Df(E^*,\theta)$ have strictly negative real parts;
\item the bifurcation produces a stable periodic orbit whose support intersects
  the relay-successor face $\mathcal{F}_{\Sigma'}$ and leaves $\mathcal{F}_\Sigma$.
\end{itemize}
\eeD

\paragraph{Structural obstruction for ME--type siphons.}
For ME--type siphons, transversal blocks $M_\sigma$ are Metzler.
Hence their dominant eigenvalue is real and cannot form a complex conjugate pair.
Therefore:

\begin{quote}
\emph{Hopf relays are impossible for purely ME--type siphons.}
\end{quote}

To allow a Hopf relay, at least one of the following must occur:
\begin{itemize}
\item the invader block $\sigma$ has dimension $\ge 2$ and includes \emph{non-Metzler}
  couplings (e.g.\ sign-changing feedback);
\item the transversal dynamics is not order preserving (e.g.\ due to behavioral or
  resource feedbacks);
\item delays or higher-order effects are introduced.
\end{itemize}

\paragraph{Equation-level difference from transcritical relays.}
In contrast to the transcritical case, where the transversal dynamics factorizes as
\[
\dot x_\sigma = M_\sigma(\bar x^*,\theta)\,x_\sigma + O(\|x_\sigma\|^2),
\]
with $M_\sigma$ Metzler, a Hopf relay requires a transversal block of the form
\[
\dot x_\sigma =
\begin{pmatrix}
a(\theta) & -\omega_0(\theta)\\
\omega_0(\theta) & a(\theta)
\end{pmatrix}x_\sigma
+ O(\|x_\sigma\|^2),
\]
or a higher-dimensional analogue, with $\omega_0(\theta)\neq 0$ at criticality.
This structure is incompatible with standard ME--type kinetics.

\subsubsection{Bogdanov--Takens relays}

\beD[Bogdanov--Takens relay]\label{def:bt-relay}
Let $(\Sigma',\Sigma)$ and $E^*\in\mathcal{F}_\Sigma$ be as above.
A \emph{Bogdanov--Takens (BT) relay} occurs if, for some parameter value $\theta=\theta_*$,
the transversal block $M_\sigma(E^*,\theta_*)$ has a double zero eigenvalue with a
one-dimensional Jordan block, while all other eigenvalues of $Df(E^*,\theta_*)$ have
negative real parts.
The resulting codimension--two bifurcation generates a relay involving equilibria
and/or periodic orbits on the relay-successor face $\mathcal{F}_{\Sigma'}$.
\eeD

\paragraph{Interpretation in the siphon lattice.}
A BT relay corresponds to a \emph{degenerate relay point} where:
\begin{itemize}
\item the invasion number satisfies $R_\sigma(E^*)=1$, and
\item its derivative along the resident manifold vanishes.
\end{itemize}
In other words, the usual first-order invasion inequality is insufficient, and
second-order terms in the invasion number become decisive.

\paragraph{Equation-level difference from transcritical relays.}
At the level of the transversal equations, a BT relay requires
\[
\dot x_\sigma = M_\sigma(\theta)\,x_\sigma + O(\|x_\sigma\|^2),
\qquad
M_\sigma(\theta_*)=
\begin{pmatrix}
0 & 1\\
0 & 0
\end{pmatrix}
\quad\text{(in suitable coordinates),}
\]
which again is impossible for generic Metzler matrices with strictly positive
off-diagonal entries.
Thus BT relays can occur only if:
\begin{itemize}
\item the invader block $\sigma$ has dimension at least two;
\item the regular splitting degenerates (loss of $M$--matrix structure);
\item additional symmetries or constraints force a higher-order contact.
\end{itemize}

\subsubsection{Comparison with transcritical relays}

\begin{center}
\begin{tabular}{|c|c|c|c|}
\hline
Relay type & Eigenvalue at $E^*$ & Generic for ME--type? & Typical outcome\\
\hline
Transcritical & Simple $\lambda=0$ & Yes & Successor equilibrium on $\mathcal{F}_{\Sigma'}$\\
Hopf & $\pm i\omega_0$ & No & Periodic orbit on $\mathcal{F}_{\Sigma'}$\\
Bogdanov--Takens & Double $\lambda=0$ & No & Equilibrium + cycle on $\mathcal{F}_{\Sigma'}$\\
\hline
\end{tabular}
\end{center}

\paragraph{Conclusion.}
Transcritical relays are structurally enforced by the order-preserving,
Metzler nature of ME--type siphons and therefore dominate the relay structure of
models such as OSN.
Hopf and Bogdanov--Takens relays require explicit violations of this structure and
should be regarded as nongeneric extensions rather than competitors of the
transcritical relay mechanism.
\section{Rank--one perturbation analysis of the symbolic Jacobian via the matrix determinant lemma}\lbl{s:R1}

In this section we apply the rank--one perturbation result to a candidate equilibrium
(e.g.\ gOSN, RFE, or an endemic equilibrium) at which the open--loop matrix $A$ is Hurwitz.
The parameter $\omega>0$ enters the Jacobian as a single rank--one feedback, and the theorem
below gives an explicit sufficient condition for Hurwitzness of the full Jacobian $J$.

Let $\mathcal F:=\{x=0\}$ be the disease--free face, and let $(0,y^\ast)\in\mathcal F$
be a disease--free equilibrium, where
\[
x=(x_1,U,W), \qquad
y=(S_1,S_2,B_1,B_2,R),
\]
and put $z=(x,y)$.

\iffalse
and the associated Sherman--Morrison formula
\[
(A+uv^\top)^{-1}
=
A^{-1}
-
\frac{A^{-1}u\,v^\top A^{-1}}{1+v^\top A^{-1}u},
\]
valid whenever $1+v^\top A^{-1}u\neq0$, which provide  explicit characterizations of how a
rank--one update affects both the spectrum and the resolvent of a matrix.
\fi

The following result is standard in the control literature-- for example see \cite{Gohberg,ZDG,FarRin}.

\beT[Stability of rank one perturbations]
\label{thm:rankone-detlemma}

Let $A\in\mathbb R^{n\times n}$ be a real matrix, $u,v\in\mathbb R^n$, and
$\kappa\in\mathbb R$. Define
\[
J := A + \kappa\, u v^\top .
\]

Assume that $A$ is Hurwitz. Define the scalar function
\[
g(\lambda):=\kappa\, v^\top(\lambda I-A)^{-1}u,
\qquad \Re\lambda\ge0,
\]
which is analytic on $\{\Re\lambda\ge0\}$ (since $A$ is Hurwitz, $\lambda I-A$ is invertible for all $\Re\lambda\ge0$).
If
\begin{equation}\label{eq:smallgain}
\sup_{\Re\lambda\ge0}|g(\lambda)|<1,
\end{equation}
then $J$ is Hurwitz.
\eeT

\begin{proof}
 Applying the matrix
determinant lemma  for rank-one updates
$
\det(B+uv^\top)=\det(B)\bigl(1+v^\top B^{-1}u\bigr)
$ (see \cite{HornJohnson}) with $B=\lambda I-A$
  yields, for every $\lambda\in\mathbb C$ such that $\lambda I-A$ is invertible,
the  factorization
\begin{equation}\label{eq:detlemma}
\det(\lambda I-J)
=
\det(\lambda I-A)\,
\Bigl(1-\kappa\, v^\top(\lambda I-A)^{-1}u\Bigr).
\end{equation}

Hence, eigenvalues of $J$ consist of those of $A$ together with the solutions of the
scalar equation
\begin{equation}\label{eq:scalarchar}
1-\kappa\, v^\top(\lambda I-A)^{-1}u=0.
\end{equation}

Since $A$ is Hurwitz, $\det(\lambda I-A)\neq0$ for all $\Re\lambda\ge0$.
By \eqref{eq:detlemma}, zeros of $\det(\lambda I-J)$ in $\Re\lambda\ge0$ can therefore
occur only if \eqref{eq:scalarchar} holds.
Condition \eqref{eq:smallgain} implies $|1-g(\lambda)|\ge1-|g(\lambda)|>0$ for all
$\Re\lambda\ge0$, hence \eqref{eq:scalarchar} has no solution in the closed right
half--plane. Thus $J$ has no eigenvalues with $\Re\lambda\ge0$ and is Hurwitz.
\end{proof}

\paragraph{Application to the OSN Jacobian: removal of the $(3,8)$ edge.}
We apply Theorem~\ref{thm:rankone-detlemma} with $\kappa=\omega$ (the model rate
parameter), $u=e_3$, and $v=e_8$ to the OSN Jacobian.
The $(3,8)$ entry $\omega$ in $J$ is the $R\to W$ feedback coupling introduced
by $\omega>0$: it is precisely the coupling identified in Section~\ref{s:Hopf}
(item~3 of the discussion of oscillation-generating modifications) as the structural
feature that breaks the block lower-triangular form and opens the door to oscillations
for $\omega>0$.  Removing it (i.e.\ setting $\kappa=0$ in $J=A+\omega e_3e_8^\top$)
restores the block lower-triangular open-loop matrix $A$, which is Hurwitz whenever
the resident equilibrium is stable in the $\omega=0$ reduction.
The OSN Jacobian is
\[
J=
\left(
\begin{array}{cccccccc}
 -\mu -\beta  U & -\beta  x_1 & 0 & 0 & 0 & 0 & 0 & 0 \\
 \beta  U & -\mu _8-W \beta _w+\beta  x_1 & -U \beta _w & 0 & 0 & 0 & 0 & 0 \\
 0 & W \beta _w & U \beta _w-\mu  & 0 & 0 & 0 & 0 & \omega \\
 0 & \mathrm{f1}_x & 0 & -\gamma _1 & 0 & \mathrm{f1}_B & 0 & 0 \\
 0 & \mathrm{f2}_x & 0 & 0 & -\gamma _2 & 0 & \mathrm{f2}_B & 0 \\
 0 & 0 & 0 & \gamma _1 & 0 & -\mu _1 & 0 & 0 \\
 0 & 0 & 0 & 0 & \gamma _2 & 0 & -\mu _2 & 0 \\
 0 & 0 & 0 & 0 & 0 & \mu _1 & \mu _2 & -\omega
\end{array}
\right).
\]

\begin{corollary}[OSN rank--one perturbation with explicit gain bound]
\label{cor:OSN-resolvent}

Consider the OSN Jacobian written as a rank--one feedback
\[
J = A + \omega\, e_3 e_8^\top,
\]
where $A:=J-\omega e_3 e_8^\top$ is an ``open--loop" Jacobian obtained by removing  the
$(3,8)$ edge and retaining the diagonal  $-\omega$.
Assume that $A$ is Metzler and Hurwitz at an equilibrium under consideration.  Then, \begin{equation}\label{eq:OSN-explicit}
|\omega|\;\bigl(-A^{-1}\bigr)_{8,3}<1.
\end{equation}
is a
sufficient condition for Hurwitzness of $J$.

\end{corollary}

\Prf\ Define, for $\Re\lambda\ge0$,
\[
h(\lambda):= e_8^\top(\lambda I-A)^{-1}e_3 .
\]

By Theorem \ref{thm:rankone-detlemma}
\begin{equation}\label{eq:OSN-smallgain}
|\omega|\, \sup_{\Re\lambda\ge0}|h(\lambda)|<1,
\end{equation}
implies that $J$ is Hurwitz and the equilibrium is linearly asymptotically stable.

Note
$(\lambda I-A)^{-1}\ge0$ entrywise for all
$\Re\lambda\ge0$, since $A$ is also Metzler.
Using the Laplace representation
\[
(\lambda I-A)^{-1}=\int_0^\infty e^{-\lambda t}e^{At}\,dt,
\qquad \Re\lambda>\alpha(A),
\]
with $\alpha(A)<0$, one obtains for all $\Re\lambda\ge0$,
\[
0\le h(\lambda)
\le \int_0^\infty [e^{At}]_{8,3}\,dt
= -\bigl(A^{-1}\bigr)_{8,3}.
\]
Consequently,
\begin{equation}\label{eq:OSN-resolvent-bound}
\sup_{\Re\lambda\ge0}
\bigl|e_8^\top(\lambda I-A)^{-1}e_3\bigr|
=h(0)= -\bigl(A^{-1}\bigr)_{8,3},
\end{equation}

\beR[DC-gain collapse]
\label{rem:DC-gain}
Because $A$ is Metzler and Hurwitz, $e^{At}\ge0$ entrywise for all $t\ge0$, so
$h(\lambda)=e_8^\top(\lambda I-A)^{-1}e_3$ is nonnegative and decreasing in
$\Re\lambda$.  Consequently the worst case over the right half-plane is attained
at $\lambda=0$, giving the \emph{DC-gain formula}
\[
\sup_{\Re\lambda\ge0}\bigl|e_8^\top(\lambda I-A)^{-1}e_3\bigr|
= h(0) = -\bigl(A^{-1}\bigr)_{8,3}.
\]
The Hurwitzness test~\eqref{eq:OSN-explicit} therefore reduces to the single
scalar inequality
\[
\omega\,\bigl(-A^{-1}\bigr)_{8,3} < 1,
\]
involving one entry of $A^{-1}$, which is computable in closed form whenever $A$ is
explicit (e.g.\ at gOSN or RFE).
\eeR

\beR[Evaluation at gOSN]
\label{rem:gOSN-eval}
At gOSN one has $S_j=B_j=R=W=0$, $\hat x_1=\mu_n/\beta$,
$\hat U=\mu(R_0-1)/\beta$, and the open-loop matrix $A=J|_{\mathrm{gOSN}}-\omega e_3e_8^\top$
coincides with the full Jacobian at gOSN (since $\omega$ only enters via the removed $(3,8)$
entry, and $R=W=0$ there independently of $\omega$).  Under the gOSN stability conditions
($R_0^W<1$, $R_{\mathrm{gOSN}}<1$), this matrix is Hurwitz (Section~\ref{s:LAS}).
The DC-gain bound~\eqref{eq:OSN-explicit} then provides an explicit upper bound on the
admissible feedback rate:
\[
\omega < \frac{1}{(-A^{-1})_{8,3}}\bigg|_{\mathrm{gOSN}},
\]
a condition that is checkable in closed form from the gOSN coordinates.
An analogous bound holds at RFE, using $\tilde x_1,\tilde U,\tilde W$ from~\eqref{eq:x1}.
\eeR

\begin{remark}[On alternative open--loop splittings]\label{rem:alt-splitting}
One may alternatively define an open loop by deleting other coupling terms (e.g.\
$\mathrm{f1}_x,\mathrm{f2}_x,\mathrm{f1}_B,\mathrm{f2}_B$), as in the approach suggested
in \cite{HA}. The determinant--lemma identity \eqref{eq:detlemma} remains valid
for any such low--rank/sparse splitting. The usefulness of a given splitting depends on
whether the corresponding open--loop matrix $A$ is Hurwitz (or at least has a spectral
structure that can be controlled) at the equilibrium under consideration.
\end{remark}

\section{Generalization of the NGM theorem: Boundary-face invariance forces $J_{x y}=0$ for polynomial mass-action fields}\label{s:NGM}
{%\Large
\beT [NGM theorem for forward invariant faces/siphons]\label{t:NGM}

(generalizing \cite{Diek,Van,Van08},\cite{JA})
Let $(x,y)\in\mathbb{R}^{m}_{\ge 0}\times\mathbb{R}^{n}_{\ge 0}$ and let
\[
\dot x = f_x(x,y),\qquad \dot y = f_y(x,y)
\]
be a polynomial positive (equivalently, mass-action representable) vector field.

Assume that the face $F_x:=\{x=0\}$ is {\bf invariant}, i.e.
\begin{equation}\label{eq:faceinv}
f_x(0,y)\equiv 0\quad\text{for all }y\in\mathbb{R}^n_{\ge 0}
\end{equation}
 (equivalently, $x$ is  a siphon).
Then:\BEN \im The mixed Jacobian block $J_{xy}:=D_y f_x(0,y)$ vanishes identically on the face $F_x$:
\[
D_y f_x(0,y)=0\quad\text{for all }y\in\mathbb{R}^n_{\ge 0},
\]
and at any point $(0,y^\ast)$ on the face, the Jacobian has block lower-triangular form
\[
J(0,y^\ast)=
\begin{pmatrix}
J_x & 0\\
* & J_y
\end{pmatrix}.
\]
\im A fixed point on the face $F_x$ is stable iff the Jacobians $J_x ,
 J_y$ at that point are both stable.
 \im If furthermore $J_x=F-V$, where $(F,V)$ is a \regS\ (\ie\ $F,V^{-1}$ are \nne),
 then a fixed point on the face $F_x$ is stable if
 $\rho(F V^{-1})<1$, and unstable if
 $\rho(F V^{-1})>1$ \cite{Varga}.

\EEN
\eeT

\begin{proof}
Fix $i\in\{1,\dots,m\}$ and write the $i$th component of $f_x$ as a polynomial
\[
(f_x)_i(x,y)=\sum_{\alpha\in\mathbb{N}^{m},\;\beta\in\mathbb{N}^{n}}
c_{\alpha,\beta}\,x^\alpha y^\beta,
\qquad
x^\alpha:=\prod_{p=1}^{m} x_p^{\alpha_p},\quad
y^\beta:=\prod_{q=1}^{n} y_q^{\beta_q},
\]
with real coefficients $c_{\alpha,\beta}$ and only finitely many nonzero terms.
Evaluating at $x=0$ yields
\[
(f_x)_i(0,y)=\sum_{\beta\in\mathbb{N}^{n}} c_{0,\beta}\,y^\beta,
\]
because every monomial with $\alpha\neq 0$ vanishes when $x=0$.

By the invariance hypothesis \eqref{eq:faceinv}, $(f_x)_i(0,y)\equiv 0$ as a polynomial in $y$ on the set
$\mathbb{R}^n_{\ge 0}$. Since a real polynomial that vanishes on a set with nonempty interior must be the zero
polynomial, it follows that $c_{0,\beta}=0$ for all $\beta$. Hence every nonzero monomial in $(f_x)_i$ has
$\alpha\neq 0$, i.e. it contains at least one factor $x_p$. Equivalently, there exists a polynomial $h_i(x,y)$ such that
\[
(f_x)_i(x,y)=\sum_{p=1}^{m} x_p\,h_{i,p}(x,y),
\]
and in particular $(f_x)_i(x,y)$ is divisible by each $x_p$ that appears in every monomial.

Now fix $j\in\{1,\dots,n\}$. Differentiating termwise gives
\[
\frac{\partial (f_x)_i}{\partial y_j}(x,y)=\sum_{\alpha,\beta} c_{\alpha,\beta}\,\beta_j\,x^\alpha y^{\beta-e_j},
\]
where $e_j$ is the $j$th standard basis vector of $\mathbb{N}^n$ and terms with $\beta_j=0$ contribute $0$.
Since every term in the sum has $\alpha\neq 0$, each monomial still contains a factor $x_p$, hence
\[
\frac{\partial (f_x)_i}{\partial y_j}(0,y)=0\qquad \text{for all }y.
\]
As $i$ and $j$ were arbitrary, this proves $D_y f_x(0,y)=0$.

Finally, the Jacobian at $(0,y^\ast)$ splits into blocks
\[
J(0,y^\ast)=
\begin{pmatrix}
D_x f_x(0,y^\ast) & D_y f_x(0,y^\ast)\\
D_x f_y(0,y^\ast) & D_y f_y(0,y^\ast)
\end{pmatrix}
=
\begin{pmatrix}
J_x & 0\\
* & J_y
\end{pmatrix},
\]
which is block lower-triangular, as claimed.
\end{proof}

\beR The difference \wrt\ the classic theorem is that we start by assuming a class of models (positive ODEs), and the essential property  needed \eqr{eq:faceinv} (forward invariance of the face investigated) in order to get the result.
\eeR

\beD [ME siphon-faces]
We will call faces where the Jacobian $J_x$  admits a regular splitting
ME siphon-faces.
\eeD

Since  the DFE-face (we recall that the DFE-face is the intersection of all the invariant faces) is a ME siphon-face for all the ME models we know, we ask:
\beO [is the DFE  face always a ME siphon-face?]
Prove or disprove that the Jacobian $J_x$ admits always a \regS\ on the DFE face.
\eeO
\section{A crucial CRN result: boundary \wli\ points always reside on critical siphons}\label{s:wl}

 Boundary attractors are indispensable in ecology, ME, and the other fields enumerated above, but it is only in CRNT that two remarkable facts about  omega-limits and forward-invariant subspaces
 have drawn sufficient attention. The first is that forward-invariant subspaces  enjoy also
 a combinatorial property known in  the Petri networks literature as being {\emph siphons/semilocking sets}
   (see  \cite{koch2010modeling} for a survey of the Petri net
approach in systems biology, and for efficient algorithms for computing them).

\beD[siphon/semilocking set, locking set]\lbl{d:sip}\cite{feinberg1987, AdLS}.
 \leavevmode

 \begin{itemize}
\item A \textbf{siphon/semilocking set} $W \subseteq \mathcal{S}$ is a nonempty subset of species such that whenever a species in $W$ appears in a product complex, at least one species in $W$ must appear in the corresponding reactant complex.

\item A \textbf{locking set} $W \subseteq \mathcal{S}$ is a nonempty subset where every reaction has at least one species from $W$ in its reactant complex.

    \im A siphon/semilocking set is \textbf{minimal}  when it contains no other siphon included within.
\end{itemize}

\eeD

\beT [siphons correspond to forward invariant faces for positive ODEs]\label{p:07} \cite[Prop. 2,Prop. 4.5]{AdLS},\cite[Prop 2.1]{ShiuStu},%\cite{AndGAS},
A set of variables $Z\subset \mathcal{S}$ in a positive ODE (where $\mathcal{S}$ is the set of all variables) is a siphon, in the sense of Petri nets literature, if and only if the associated boundary face
$\textrm{cl}(F_Z)$ where all variables in $Z$ are $0$ is forward invariant for the ODE.
\eeT

This shows that the non-structural concept of forward-invariant boundary faces (\ie\ dependent on the dynamics, apparently), is in fact   structural (\ie\ it may be defined using only the rates independent part of its CRN definition, the ``reactions network").

%Further results for mass action systems were provided by Shiu and Sturmfels \cite[Lem. 2.3-2.4]{ShiuStu}.

The  semi-locking  property is  fundamental for understanding persistence.
\beD[persistence]\label{d:per} For a mass--action system  on $\mathbb{R}^n_{\ge 0}$, \emph{persistence} means trajectories starting in the positive orthant remain bounded away from the boundary: if $x_i(0)>0$ for all $i$, then $\liminf_{t\to\infty}x_i(t)>0$ for all $i$ \cite{AdLS}.
\eeD

 %\cite{AdLS} showed that if the $\omega$-limit set does not intersect semi-locking boundaries except at equilibria, and if all trajectories starting on non-semilocking boundaries eventually leave those boundaries,  then the system exhibits persistence.

 \beT [Boundary omega-limit points always reside on critical siphons]\cite[Prop. 1]{AdLS},\cite[Lem. 63]{FreFelW}%\cite{AndGAS},
Consider a reaction network such that
\BI \im (r1) the rates are nonnegative, differentiable, \im  (r2) the rates are $0$ whenever some reactant (these are the variables of each rate) is $0$, and
\im (r3) the rates are strictly increasing when the reactants are positive.
\EI
 Then, for any boundary \wli\ point $E \in \omega(\xi), \xi \in \mathcal{O}^+$ letting $Z(E)$ denote the  zero coordinate set of $E$, it holds that
 $Z(E)$ is a siphon.
\eeT

\beR In other words, any boundary $\omega$-limit point $E$ must lie on a  siphon face, or the  zero coordinate set $Z(E)$  must be a siphon. Note however
that the siphon concept is structural, while that of $\omega$-limit point isn't, being seemingly parameter dependent.  However, this result shows that $\omega$-limit points, just like invariant faces, are a structural concept.
\eeR

\beR
Later, \cite[Lem. 63]{FreFelW} reproved this amazing result with a weaker (r3)
``the flow  is forward-complete; in other words, for any initial state, the (unique) maximal solution of the corresponding initial value problem in is defined for all $t \geq 0$".

\eeR

We can think informally of siphons as of ``shelters for boundary equilibria", and the above result suggests   preceding  the identification of all \bfp s  by the identification of all siphons,  the latter being an easier problem, which does not depend on the precise nature of its rates (polynomial, algebraic, or non-algebraic), as long as they \saty\ the conditions of  \cite{AdLS} generalized in \cite{FreFelW}. These conditions  are satisfied under the most common kinetic assumptions in the literature, namely, mass-action,  power-law kinetics, Michaelis-Menten kinetics, or Hill kinetics, as well as combinations of these \cite[pg 585--586]{AdLS10}.

\beR \label{r:rec} Another important tool for positive ODEs worth highlighting comes from the generalized Lotka-Volterra theory in ecology, which may be viewed as focusing on the particular case in which each species is a siphon, and  whose equilibria \saty\ a polynomial complementarity problem \cite{Cui2025Comp}.

Solving Lotka-Volterra problems could possibly be integrated in symbolic and numeric packages for the solution of polynomial systems, since it enjoys a recursive structure (for each variable, either is is 0, and we are left with a system of  dimension smaller by one,  or this variable can be simplified) but this does not seem to have happened yet.

\eeR
\beR \label{r:lat}A moment of thought reveals that a similar recursion holds place for systems whose siphons are precisely the lattice generated by minimal siphons.
\eeR

As of recently,  by generalizing {pairwise} interactions  to group ones,  classic tools like graphs, adjacency matrices  and Metzler matrices (whose off-diagonal elements are all non-negative, and  are commonly associated with cooperative systems, such as generalized Lotka-Volterra, and single-virus systems \cite{liu2020stability,cui2022discrete,pare2018analysis}) have been extended  to \emph{hypergraphs}, \emph{adjacency tensors} and \emph{Metzler tensors} \cite{gallo1993directed,chang2013survey,chang2008perron,
yang2010further,yang2011further,bick2023higher}. Pioneering works like
\cite{Letten,iacopini2019simplicial,cisneros2021multigroup,cui2023general,
cui2025metzler,cui2025analysis} have initiated the study of  positive ODES
modeling group interactions in ecology and epidemiology. The approach of these authors
 suggests that, despite their different focuses, a unification of BIN subfields is profitable.

\section*{Declarations}

\textbf{Complience of Ethical Standard:} The authors have no competing interests to declare that are relevant to the content of this article. The paper is of a theoretical nature so 
no experiments were performed and no human or animal subjects were involved.

\iffalse
\section{Further review of  useful facts about positive/essentially non-negative ODEs}
The \BIN\  disciplines  are all concerned with positive  dynamical systems, but neglect sometimes important results in their  ``sister sciences", and one of our motivations here is to increase the awareness to general important results which have been cited sometimes thousands of times in one field, and almost never in another.  To emphasize their importance, we will call them in the next subsection laws.
\fi

\bibliographystyle{apalike}%{amsalpha}

%\bibliography{biblio,biblio_Arino_Julien,epidemio,math_epi,math,ref}
 
\end{document}